\documentclass{birkmult}
\usepackage{amsmath}
\usepackage{amsxtra}
\usepackage{amscd}
\usepackage{amsthm}
\usepackage{amsfonts}
\usepackage{amssymb}
\usepackage{eucal}

\newtheorem{cor}[subsection]{Corollary}

\newtheorem{lem}[subsection]{Lemma}
\newtheorem{prop}[subsection]{Proposition}

\newtheorem{conj}[subsection]{Conjecture}
\newtheorem{thm}[subsection]{Theorem}
\newtheorem{defn}[subsection]{Definition}
\newtheorem{rem}[subsection]{Remark}

\theoremstyle{definition}

\theoremstyle{remark}

\newcommand{\nc}{\newcommand}
\nc{\renc}{\renewcommand} \nc{\ssec}{\subsection}
\nc{\sssec}{\subsubsection} \nc{\on}{\operatorname}

\nc\ol{\overline} \nc\ul{\underline} \nc\wt{\widetilde}
\nc\tboxtimes{\wt{\boxtimes}} \nc{\alp}{\alpha}

\nc{\ZZ}{{\mathbb Z}} \nc{\NN}{{\mathbb N}} \nc{\CC}{{\mathbb C}}
\nc{\OO}{{\mathbb O}} \renc{\SS}{{\mathbb S}} \nc{\DD}{{\mathbb
D}}
\renewcommand{\AA}{{\mathbb A}}
\nc{\Fq}{{\mathbb F}_q} \nc{\Fqb}{\ol{{\mathbb F}_q}}
\nc{\Ql}{\ol{{\mathbb Q}_\ell}} \nc{\id}{\text{id}} \nc\X{\mathcal
X}

\nc{\Hom}{\on{Hom}} \nc{\Lie}{\on{Lie}} \nc{\Loc}{\on{Loc}}
\nc{\Pic}{\on{Pic}} \nc{\Bun}{\on{Bun}} \nc{\IC}{\on{IC}}
\nc{\Aut}{\on{Aut}} \nc{\rk}{\on{rk}} \nc{\Sh}{\on{Sh}}
\nc{\Perv}{\on{Perv}} \nc{\pos}{{\on{pos}}} \nc{\Conv}{\on{Conv}}
\nc{\Sph}{\on{Sph}} \nc{\Sym}{\on{Sym}}
\nc{\BunBb}{\overline{\Bun}_B} \nc{\Buno}{\overset{o}{\Bun}}
\nc{\BunPb}{{\overline{\Bun}_P}}
\nc{\BunBM}{\overline{\Bun}_{B(M)}}
\nc{\BunPbw}{{\widetilde{\Bun}_P}}
\nc{\BunBP}{\widetilde{\Bun}_{B,P}} \nc{\GUb}{\overline{G/U}}
\nc{\GUPb}{\overline{G/U(P)}}
\nc{\iso}{{\stackrel{\sim}{\longrightarrow}}}

\nc{\Hhom}{\underline{\on{Hom}}} \nc\syminfty{\on{Sym}^{\infty}}
\nc\lal{\ol{\lambda}} \nc\xl{\ol{x}} \nc\thl{\ol{\theta}}
\nc\nul{\ol{\nu}} \nc\mul{\ol{\mu}} \nc\Sum\Sigma
\nc{\oX}{\overset{o}{X}{}}

\nc{\M}{{\mathcal M}} \nc{\N}{{\mathcal N}} \nc{\F}{{\mathcal F}}
\nc{\D}{{\mathcal D}} \nc{\Q}{{\mathcal Q}} \nc{\Y}{{\mathcal Y}}
\nc{\G}{{\mathcal G}} \nc{\E}{{\mathcal E}} \nc{\CalC}{{\mathcal
C}}
\nc\Dh{\widehat{\D}}
\renewcommand{\O}{{\mathcal O}}
\nc{\C}{{\mathcal C}} \nc{\K}{{\mathcal K}}
\renewcommand{\H}{{\mathcal H}}

\nc{\T}{{\mathcal T}} \nc{\V}{{\mathcal V}} \renc{\P}{{\mathcal
P}} \nc{\A}{{\mathcal A}} \nc{\B}{{\mathcal B}} \nc{\U}{{\mathcal
U}}
\renewcommand{\L}{{\mathcal L}}
\nc{\Gr}{\on{Gr}}

\nc{\frn}{{\check{\mathfrak u}(P)}}
\nc\f{{\mathfrak f}}

\nc{\q}{{\mathfrak q}} \nc{\p}{{\mathfrak p}} \nc{\s}{{\mathfrak
s}} \nc\w{\text{w}}

\nc\mathi\iota \nc\Spec{\on{Spec}} \nc\Mod{\on{Mod}}
\nc{\tw}{\widetilde{\mathfrak t}} \nc{\pw}{\widetilde{\mathfrak
p}} \nc{\qw}{\widetilde{\mathfrak q}} \nc{\jw}{\widetilde j}

\nc{\grb}{\overline{\Gr}} \nc{\I}{\mathcal I}

\nc{\lambdach}{{\check\lambda}} \nc{\Lambdach}{{\check\Lambda}{}}
\nc{\much}{{\check\mu}} \nc{\omegach}{{\check\omega}}
\nc{\nuch}{{\check\nu}} \nc{\etach}{{\check\eta}}
\nc{\alphach}{{\check\alpha}} \nc{\betach}{{\check\beta}}
\nc{\rhoch}{{\check\rho}} \nc{\ch}{{\check h}}

\nc{\Hb}{\overline{\H}}

\nc{\BA}{{\mathbb{A}}} \nc{\BC}{{\mathbb{C}}}
\nc{\BM}{{\mathbb{M}}} \nc{\BN}{{\mathbb{N}}}
\nc{\BP}{{\mathbb{P}}} \nc{\BR}{{\mathbb{R}}}
\nc{\BZ}{{\mathbb{Z}}} \nc{\BS}{{\mathbb{S}}}

\nc{\CA}{{\mathcal{A}}} \nc{\CB}{{\mathcal{B}}}
\nc{\CE}{{\mathcal{E}}} \nc{\CF}{{\mathcal{F}}}
\nc{\CG}{{\mathcal{G}}} \nc{\CH}{{\mathcal{H}}}
\nc{\CI}{{\mathcal{I}}} \nc{\CL}{{\mathcal{L}}}
\nc{\CM}{{\mathcal{M}}} \nc{\CN}{{\mathcal{N}}}
\nc{\CO}{{\mathcal{O}}} \nc{\CP}{{\mathcal{P}}}
\nc{\CQ}{{\mathcal{Q}}} \nc{\CR}{{\mathcal{R}}}
\nc{\CS}{{\mathcal{S}}} \nc{\CT}{{\mathcal{T}}}
\nc{\CU}{{\mathcal{U}}} \nc{\CV}{{\mathcal{V}}}
\nc{\CW}{{\mathcal{W}}} \nc{\CZ}{{\mathcal{Z}}}

\nc{\cM}{{\check{\mathcal M}}{}} \nc{\csM}{{\check{\mathcal A}}{}}
\nc{\oM}{{\overset{\circ}{\mathcal M}}{}}
\nc{\obM}{{\overset{\circ}{\mathbf M}}{}}
\nc{\oCA}{{\overset{\circ}{\mathcal A}}{}}
\nc{\obA}{{\overset{\circ}{\mathbf A}}{}}
\nc{\ooM}{{\overset{\circ}{M}}{}}
\nc{\osM}{{\overset{\circ}{\mathsf M}}{}}
\nc{\vM}{{\overset{\bullet}{\mathcal M}}{}}
\nc{\nM}{{\underset{\bullet}{\mathcal M}}{}}
\nc{\oD}{{\overset{\circ}{\mathcal D}}{}}
\nc{\obD}{{\overset{\circ}{\mathbf D}}{}}
\nc{\oA}{{\overset{\circ}{\mathbb A}}{}}
\nc{\op}{{\overset{\bullet}{\mathbf p}}{}}
\nc{\cp}{{\overset{\circ}{\mathbf p}}{}}
\nc{\oU}{{\overset{\bullet}{\mathcal U}}{}}
\nc{\oZ}{{\overset{\circ}{\mathcal Z}}{}}
\nc{\ofZ}{{\overset{\circ}{\mathfrak Z}}{}}

\nc{\ff}{{\mathfrak{f}}} \nc{\fv}{{\mathfrak{v}}}
\nc{\fa}{{\mathfrak{a}}} \nc{\fb}{{\mathfrak{b}}}
\nc{\fd}{{\mathfrak{d}}} \nc{\fe}{{\mathfrak{e}}}
\nc{\fg}{{\mathfrak{g}}} \nc{\fgl}{{\mathfrak{gl}}}
\nc{\fh}{{\mathfrak{h}}} \nc{\fri}{{\mathfrak{i}}}
\nc{\fj}{{\mathfrak{j}}} \nc{\fk}{{\mathfrak{k}}}
\nc{\fm}{{\mathfrak{m}}} \nc{\fn}{{\mathfrak{n}}}
\nc{\ft}{{\mathfrak{t}}} \nc{\fu}{{\mathfrak{u}}}
\nc{\fw}{{\mathfrak{w}}} \nc{\fz}{{\mathfrak{z}}}
\nc{\fp}{{\mathfrak{p}}} \nc{\frr}{{\mathfrak{r}}}
\nc{\fs}{{\mathfrak{s}}} \nc{\fsl}{{\mathfrak{sl}}}
\nc{\hsl}{{\widehat{\mathfrak{sl}}}}
\nc{\hgl}{{\widehat{\mathfrak{gl}}}}
\nc{\hg}{{\widehat{\mathfrak{g}}}}
\nc{\chg}{{\widehat{\mathfrak{g}}}{}^\vee}
\nc{\hn}{{\widehat{\mathfrak{n}}}}
\nc{\chn}{{\widehat{\mathfrak{n}}}{}^\vee}

\nc{\fA}{{\mathfrak{A}}} \nc{\fB}{{\mathfrak{B}}}
\nc{\fD}{{\mathfrak{D}}} \nc{\fE}{{\mathfrak{E}}}
\nc{\fF}{{\mathfrak{F}}} \nc{\fG}{{\mathfrak{G}}} \nc{\fH}{{\mathfrak{H}}}
\nc{\fI}{{\mathfrak{I}}} \nc{\fJ}{{\mathfrak{J}}}
\nc{\fK}{{\mathfrak{K}}} \nc{\fL}{{\mathfrak{L}}}
\nc{\fM}{{\mathfrak{M}}} \nc{\fN}{{\mathfrak{N}}}
\nc{\frP}{{\mathfrak{P}}} \nc{\fQ}{{\mathfrak{Q}}}
\nc{\fT}{{\mathfrak{T}}} \nc{\fU}{{\mathfrak{U}}}
\nc{\fV}{{\mathfrak{V}}} \nc{\fW}{{\mathfrak{W}}}
\nc{\fX}{{\mathfrak{X}}} \nc{\fY}{{\mathfrak{Y}}}
\nc{\fZ}{{\mathfrak{Z}}}

\nc{\ba}{{\mathbf{a}}} \nc{\bb}{{\mathbf{b}}} \nc{\bc}{{\mathbf{c}}}
\nc{\be}{{\mathbf{e}}}
\nc{\bff}{{\mathbf{f}}}\nc{\bk}{{\mathbf{k}}}\nc{\bj}{{\mathbf{j}}}
\nc{\bn}{{\mathbf{n}}} \nc{\bp}{{\mathbf{p}}} \nc{\bq}{{\mathbf{q}}}
\nc{\bs}{{\mathbf{s}}} \nc{\bfu}{{\mathbf{u}}}
\nc{\bv}{{\mathbf{v}}} \nc{\bx}{{\mathbf{x}}} \nc{\by}{{\mathbf{y}}}
\nc{\bw}{{\mathbf{w}}} \nc{\bA}{{\mathbf{A}}} \nc{\bB}{{\mathbf{B}}}
\nc{\bC}{{\mathbf{C}}} \nc{\bD}{{\mathbf{D}}} \nc{\bF}{{\mathbf{F}}}
\nc{\bH}{{\mathbf{H}}} \nc{\bK}{{\mathbf{K}}} \nc{\bM}{{\mathbf{M}}}
\nc{\bN}{{\mathbf{N}}} \nc{\bO}{{\mathbf{O}}} \nc{\bS}{{\mathbf{S}}}
\nc{\bV}{{\mathbf{V}}} \nc{\bW}{{\mathbf{W}}} \nc{\bX}{{\mathbf{X}}}
\nc{\bY}{{\mathbf{Y}}} \nc{\bP}{{\mathbf{P}}} \nc{\bZ}{{\mathbf{Z}}}
\nc{\bh}{{\mathbf{h}}}

\nc{\sA}{{\mathsf{A}}} \nc{\sB}{{\mathsf{B}}}
\nc{\sC}{{\mathsf{C}}} \nc{\sD}{{\mathsf{D}}}
\nc{\sE}{{\mathsf{E}}} \nc{\sF}{{\mathsf{F}}}
\nc{\sK}{{\mathsf{K}}} \nc{\sL}{{\mathsf{L}}}
\nc{\sM}{{\mathsf{M}}} \nc{\sO}{{\mathsf{O}}}
\nc{\sQ}{{\mathsf{Q}}} \nc{\sP}{{\mathsf{P}}}
\nc{\sT}{{\mathsf{T}}} \nc{\sZ}{{\mathsf{Z}}}
\nc{\sV}{{\mathsf{V}}}
\nc{\sfp}{{\mathsf{p}}} \nc{\sr}{{\mathsf{r}}}
\nc{\st}{{\mathsf{t}}} \nc{\sfb}{{\mathsf{b}}}
\nc{\sfc}{{\mathsf{c}}} \nc{\sd}{{\mathsf{d}}}
\nc{\sz}{{\mathsf{z}}}

\nc{\BK}{{\bar{K}}}

\nc{\tA}{{\widetilde{\mathbf{A}}}}
\nc{\tB}{{\widetilde{\mathcal{B}}}}
\nc{\tg}{{\widetilde{\mathfrak{g}}}} \nc{\tG}{{\widetilde{G}}}
\nc{\TM}{{\widetilde{\mathbb{M}}}{}}
\nc{\tO}{{\widetilde{\mathsf{O}}}{}}
\nc{\tU}{{\widetilde{\mathfrak{U}}}{}} \nc{\TZ}{{\tilde{Z}}}
\nc{\tx}{{\tilde{x}}} \nc{\tbv}{{\tilde{\bv}}}
\nc{\tfP}{{\widetilde{\mathfrak{P}}}{}} \nc{\tz}{{\tilde{\zeta}}}
\nc{\tmu}{{\tilde{\mu}}}

\nc{\urho}{\underline{\rho}} \nc{\uB}{\underline{B}}
\nc{\uC}{{\underline{\mathbb{C}}}} \nc{\ui}{\underline{i}}
\nc{\uj}{\underline{j}} \nc{\ofP}{{\overline{\mathfrak{P}}}}
\nc{\oB}{{\overline{\mathcal{B}}}}
\nc{\og}{{\overline{\mathfrak{g}}}} \nc{\oI}{{\overline{I}}}

\nc{\eps}{\varepsilon} \nc{\hrho}{{\hat{\rho}}}
\nc{\blambda}{{\boldsymbol{\lambda}}}

\nc{\one}{{\mathbf{1}}} \nc{\two}{{\mathbf{t}}}

\nc{\Rep}{{\mathop{\operatorname{\rm Rep}}}}
\nc{\Tot}{{\mathop{\operatorname{\rm Tot}}}}
\nc{\Ker}{{\mathop{\operatorname{\rm Ker}}}}
\nc{\Hilb}{{\mathop{\operatorname{\rm Hilb}}}}
\nc{\End}{{\mathop{\operatorname{\rm End}}}}
\nc{\Ext}{{\mathop{\operatorname{\rm Ext}}}}
\nc{\CHom}{{\mathop{\operatorname{{\mathcal{H}}\it om}}}}
\nc{\GL}{{\mathop{\operatorname{\rm GL}}}}
\nc{\gr}{{\mathop{\operatorname{\rm gr}}}}
\nc{\Id}{{\mathop{\operatorname{\rm Id}}}}
\nc{\defi}{{\mathop{\operatorname{\rm def}}}}
\nc{\length}{{\mathop{\operatorname{\rm length}}}}
\nc{\supp}{{\mathop{\operatorname{\rm supp}}}}

\nc{\Cliff}{{\mathsf{Cliff}}}
\nc{\Fl}{{\mathsf{Fl}}} \nc{\Fib}{{\mathsf{Fib}}}
\nc{\Coh}{{\mathsf{Coh}}} \nc{\FCoh}{{\mathsf{FCoh}}}

\nc{\reg}{{\text{\rm reg}}}

\nc{\cplus}{{\mathbf{C}_+}} \nc{\cminus}{{\mathbf{C}_-}}
\nc{\cthree}{{\mathbf{C}_*}} \nc{\Qbar}{{\bar{Q}}}

\nc{\bOmega}{{\overline{\Omega}}}

\nc{\seq}[1]{\stackrel{#1}{\sim}}

\nc{\aff}{\operatorname{aff}}

\def\dsp{\displaystyle}

\begin{document}

\title[$K$-theory of Laumon spaces]
 {Quantum affine Gelfand-Tsetlin bases and quantum toroidal algebra via
 $K$-theory of affine Laumon spaces}

\author[A. Tsymbaliuk]{Alexander Tsymbaliuk}
 \address{Independent University of Moscow, 11 Bol'shoy Vlas'evskiy per., Moscow 119002, Russia}
 \curraddr{Department of Mathematics, MIT, 77 Mass. Ave., Cambridge, MA  02139, USA}
\email{sasha\_ts@mit.edu}

\begin{abstract}
Laumon moduli spaces are certain smooth closures of the moduli
spaces of maps from the projective line to the flag variety of
$GL_n$. We construct the action of the quantum loop algebra
$U_v(\textbf{L}\fsl_n)$ in the $K$-theory of Laumon spaces by
certain natural correspondences. Also we construct the action of the
quantum toroidal algebra \"{U}$_v(\widehat{\fsl}_n)$ in the
$K$-theory of the affine version of Laumon spaces.
\end{abstract}

\maketitle

\section{Introduction}
This note is a sequel to~\cite{fr, ffnr}. The moduli spaces
$\fQ_{\ul{d}}$ were introduced by G.~Laumon in~\cite{la1}
and~\cite{la2}. They are certain partial compactifications of the
moduli spaces of degree $\ul{d}$ based maps from $\BP^1$ to the flag
variety $\CB_n$ of $GL_n$.
 The authors of~\cite{fr,ffnr} considered the localized equivariant
cohomology
$R=\bigoplus_{\ul{d}}H^\bullet_{\widetilde{T}\times\BC^*}(\fQ_{\ul{d}})
\otimes_{H^\bullet_{\widetilde{T}\times\BC^*}(pt)}\on{Frac}
(H^\bullet_{\widetilde{T}\times\BC^*}(pt))$ where $\widetilde T$ is
a Cartan torus of $GL_n$ acting naturally on the target $\CB_n$, and
$\BC^*$ acts as ``loop rotations'' on the source $\BP^1$.
 They constructed the action of the Yangian $Y(\fsl_n)$ on $R$, the new Drinfeld
generators acting by natural correspondences.

 In this note we write (in style of~\cite{ffnr}) the formulas for the action of ''Drinfeld generators''
of the quantum loop algebra in the localized equivariant $K$-theory
$M=\bigoplus_{\ul{d}}K^{\widetilde{T}\times\BC^*}(\fQ_{\ul{d}})
\otimes_{K^{\widetilde{T}\times\BC^*}(pt)}\on{Frac}
(K^{\widetilde{T}\times\BC^*}(pt))$. In fact, the correspondences
defining this action are very similar to the correspondences used by
H.~Nakajima~\cite{nak-quiver} to construct the action of the quantum
loop algebra in the equivariant $K$-theory of quiver varieties.

 We prove the main theorem directly by checking all relations in the
fixed point basis.

There is an affine version of the Laumon spaces, namely the moduli
spaces $\CP_{\ul{d}}$ of parabolic sheaves on $\BP^1\times\BP^1$, a
certain partial compactification of the moduli spaces of degree $d$
based maps from $\BP^1$ to the "thick" flag variety of the loop
group $\widehat {SL_n}$, see~\cite{fgk}. The similar correspondences
give rise to an action of the quantum toroidal algebra
\"{U}$_v(\widehat{\fsl}_n)$ on the sum of localized equivariant
$K$-groups
$V=\bigoplus_{\ul{d}}K^{\widetilde{T}\times\BC^*\times\BC^*}
(\CP_{\ul{d}})
\otimes_{K^{\widetilde{T}\times\BC^*\times\BC^*}(pt)}\on{Frac}
(K^{\widetilde{T}\times\BC^*\times\BC^*}(pt))$ where the second copy
of $\BC^*$ acts by the loop rotation on the second copy of $\BP^1$
(Theorem~\ref{varaf}).

Since the fixed point basis of $M$ corresponds to the
Gelfand-Tsetlin basis of the universal Verma module over
$U_v(\fgl_n)$ (Theorem 6.3 in~\cite{fr}), we propose to call the
fixed point basis of $V$ the {\em affine Gelfand-Tsetlin basis}. We
expect that the specialization of the affine Gelfand-Tsetlin basis
gives rise to a basis in the integrable
$U_v(\widehat{\fgl}_n)$-modules (which we also propose to call the
affine Gelfand-Tsetlin basis). We expect (see 4.17) that the action
of \"{U}$_v(\widehat{\fsl}_n)$ on the integrable
$U_v(\widehat{\fgl}_n)$-modules coincides with the action of Uglov
and Takemura~\cite{u1}. It seems likely that these
\"{U}$_v(\widehat{\fsl}_n)$--modules are obtained by the application
of the {\em Schur} functor (\cite{gkv}) to the irreducible
$\mathfrak X$-semisimple modules over the double affine Cherednik
algebra \"{H}$_n(v)$ of type $A_{n-1}$, see~\cite{sv}.

\subsection{Acknowledgments}
 I am highly indebted to Boris Feigin and Michael Finkelberg for teaching me remarkable mathematics, for introducing to this
topic and for frequent stimulating discussions.
 I am grateful to Alexander Molev
for some useful remarks concerning $q$-Yangians.


\section{Laumon spaces and quantum loop algebra $U_q(\textbf{L}\fsl_n)$ }

\subsection{Laumon spaces}
We recall the setup of ~\cite{bf,fr,ffnr}. Let $\bC$ be a smooth
projective curve of genus zero. We fix a coordinate $z$ on $\bC$,
and consider the action of $\BC^*$ on $\bC$ such that
$v(z)=v^{-2}z$. We have $\bC^{\BC^*}=\{0,\infty\}$.

We consider an $n$-dimensional vector space $W$ with a basis
$w_1,\ldots,w_n$. This defines a Cartan torus $T\subset
G=GL_n\subset Aut(W)$. We also consider its $2^n$-fold cover,
the bigger torus $\widetilde{T}$, acting on $W$ as follows: for
$\widetilde{T}\ni\ul{t}=(t_1,\ldots,t_n)$ we have
$\ul{t}(w_i)=t_i^2w_i$. We denote by $\CB$ the flag variety of
$G$.

Given an $(n-1)$-tuple of nonnegative integers
$\ul{d}=(d_1,\ldots,d_{n-1})$, we consider the Laumon's
quasiflags' space $\CQ_{\ul{d}}$, see ~\cite{la2}, ~4.2. It is the
moduli space of flags of locally free subsheaves
$$0\subset\CW_1\subset\cdots\subset\CW_{n-1}\subset\CW=W\otimes\CO_\bC$$
such that $\on{rank}(\CW_k)=k$, and $\deg(\CW_k)=-d_k$. It is known
to be a smooth projective variety of dimension
$2d_1+\cdots+2d_{n-1}+\dim\CB$, see ~\cite{la1}, ~2.10.

 We consider the following locally closed subvariety
$\fQ_{\ul{d}}\subset\CQ_{\ul{d}}$ (quasiflags based at
$\infty\in\bC$) formed by the flags
$$0\subset\CW_1\subset\cdots\subset\CW_{n-1}\subset\CW=W\otimes\CO_\bC$$
such that $\CW_i\subset\CW$ is a vector subbundle in a neighbourhood
of $\infty\in\bC$, and the fiber of $\CW_i$ at $\infty$ equals the
span $\langle w_1,\ldots,w_i\rangle\subset W$. It is known to be a
smooth quasiprojective variety of dimension $2d_1+\cdots+2d_{n-1}$.

\subsection{Fixed points}
\label{fixed points} The group $G\times\BC^*$ acts naturally on
$\CQ_{\ul{d}}$, and the group $\widetilde{T}\times\BC^*$ acts
naturally on $\fQ_{\ul{d}}$. The set of fixed points of
$\widetilde{T}\times\BC^*$ on $\fQ_{\ul{d}}$ is finite; we recall
its description from ~\cite{fk}, ~2.11.

Let $\widetilde{\ul{d}}$ be a collection of nonnegative integers
$(d_{ij}),\ i\geq j$, such that $d_i=\sum_{j=1}^id_{ij}$, and for
$i\geq k\geq j$ we have $d_{kj}\geq d_{ij}$. Abusing notation we
denote by $\widetilde{\ul{d}}$ the corresponding
$\widetilde{T}\times\BC^*$-fixed point in $\fQ_{\ul{d}}$:

$\CW_1=\CO_\bC(-d_{11}\cdot0)w_1,$

$\CW_2=\CO_\bC(-d_{21}\cdot0)w_1\oplus\CO_\bC(-d_{22}\cdot0)w_2,$

$\vdots$

$\CW_{n-1}=\CO_\bC(-d_{n-1,1}\cdot0)w_1\oplus\CO_\bC(-d_{n-1,2}\cdot0)w_2
\oplus\cdots\oplus\CO_\bC(-d_{n-1,n-1}\cdot0)w_{n-1}.$

\medskip

 \emph{Notation:} Given a collection $\widetilde{\ul{d}}$ as above, we will
denote by $\widetilde{\ul{d}}+\delta_{i,j}$ the collection
$\widetilde{\ul{d}}{}'$, such that
$\widetilde{\ul{d}}{}'_{i,j}=\widetilde{\ul{d}}_{i,j}+1$, while
$\widetilde{\ul{d}}{}'_{p,q}=\widetilde{\ul{d}}_{p,q}$ for $(p,q)\ne
(i,j)$ (in all our cases it will satisfy the required conditions,
though in general as defined it might not).

\subsection{Correspondences}
\label{classic} For $i\in\{1,\ldots,n-1\}$, and
$\ul{d}=(d_1,\ldots,d_{n-1})$, we set
$\ul{d}+i:=(d_1,\ldots,d_i+1,\ldots,d_{n-1})$. We have a
correspondence $\sE_{\ul{d},i}\subset\CQ_{\ul{d}}\times
\CQ_{\ul{d}+i}$ formed by the pairs $(\CW_\bullet,\CW'_\bullet)$
such that for $j\ne i$ we have $\CW_j=\CW'_j$, and
$\CW'_i\subset\CW_i$, see ~\cite{fk}, ~3.1. In other words,
$\sE_{\ul{d},i}$ is the moduli space of flags of locally free
sheaves
$$0\subset\CW_1\subset\cdots\subset\CW_{i-1}\subset\CW'_i\subset\CW_i\subset
\CW_{i+1}\subset\cdots\subset\CW_{n-1}\subset\CW$$ such that
$\on{rank}(\CW_k)=k$ and $\deg(\CW_k)=-d_k$, while
$\on{rank}(\CW'_i)=i$ and $\deg(\CW'_i)=-d_i-1$.

According to ~\cite{la1}, ~2.10, $\sE_{\ul{d},i}$ is a smooth
projective algebraic variety of dimension
$2d_1+\cdots+2d_{n-1}+\dim\CB+1$.

We denote by $\bp$ (resp. $\bq$) the natural projection
$\sE_{\ul{d},i}\to\CQ_{\ul{d}}$ (resp.
$\sE_{\ul{d},i}\to\CQ_{\ul{d}+i}$). We also have a map $\bs:\
\sE_{\ul{d},i}\to\bC,$
$$(0\subset\CW_1\subset\cdots\subset\CW_{i-1}\subset\CW'_i\subset\CW_i\subset
\CW_{i+1}\subset\cdots\subset\CW_{n-1}\subset\CW)\mapsto\on{supp}(\CW_i/\CW'_i).$$

The correspondence $\sE_{\ul{d},i}$ comes equipped with a natural
line bundle $L_i$ whose fiber at a point
$$(0\subset\CW_1\subset\cdots\subset\CW_{i-1}\subset\CW'_i\subset\CW_i\subset
\CW_{i+1}\subset\cdots\subset\CW_{n-1}\subset\CW)$$ equals
$\Gamma(\bC,\CW_i/\CW'_i)$. Finally, we have a transposed
correspondence $^\sT\sE_{\ul{d},i}\subset
\CQ_{\ul{d}+i}\times\CQ_{\ul{d}}$.

Restricting to $\fQ_{\ul{d}}\subset\CQ_{\ul{d}}$ we obtain the
correspondence
$\sE_{\ul{d},i}\subset\fQ_{\ul{d}}\times\fQ_{\ul{d}+i}$ together
with the line bundle $L_i$ and the natural maps $\bp:\
\sE_{\ul{d},i}\to\fQ_{\ul{d}},\ \bq:\
\sE_{\ul{d},i}\to\fQ_{\ul{d}+i},\ \bs:\
\sE_{\ul{d},i}\to\bC\backslash\{\infty\}$. We also have a transposed
correspondence $^\sT\sE_{\ul{d},i}\subset
\fQ_{\ul{d}+i}\times\fQ_{\ul{d}}$. It is a smooth quasiprojective
variety of dimension $2d_1+\ldots+2d_{n-1}+1$.

\subsection{Equivariant $K$-groups}
We denote by ${}'M$ the direct sum of equivariant (complexified)
$K$-groups:
$${}'M=\oplus_{\ul{d}}K^{\widetilde{T}\times\BC^*}(\fQ_{\ul{d}}).$$ It
is a module over
$K^{\widetilde{T}\times\BC^*}(pt)=\BC[T\times\BC^{*}]=
\BC[x_1,\ldots,x_n,v]$. We define $$M=\
{}'M\otimes_{K^{\widetilde{T}\times\BC^*}(pt)}
\on{Frac}(K^{\widetilde{T}\times\BC^*}(pt)).$$

We have an evident grading $$M=\oplus_{\ul{d}}M_{\ul{d}},\
M_{\ul{d}}=K^{\widetilde{T}\times\BC^*}(\fQ_{\ul{d}})
\otimes_{K^{\widetilde{T}\times\BC^*}(pt)}
\on{Frac}(K^{\widetilde{T}\times\BC^*}(pt)).$$

\subsection{Quantum universal enveloping algebra $U_v(\fgl_n)$}
For the quantum universal enveloping algebra $U_v(\fgl_n)$ we follow
the notations of section~2 of~\cite{mtz}. Namely, $U_v(\fgl_n)$ has
generators $\ft_1^{\pm 1},\ldots, \ft_n^{\pm
1},\fe_1,\ldots,\fe_{n-1},\ff_1,\ldots,\ff_{n-1}$ with the following
defining relations (formulas (2.1) of {\em loc. cit.}):

\begin{equation}
\label{quantgl1} \ft_i\ft_j=\ft_j\ft_i,\
\ft_i\ft_i^{-1}=\ft_i^{-1}\ft_i=1
\end{equation}

\begin{equation}
\label{quantgl2}
\ft_i\fe_j\ft_i^{-1}=\fe_jv^{\delta_{i,j}-\delta_{i,j+1}},\
\ft_i\ff_j\ft_i^{-1}=\ff_jv^{-\delta_{i,j}+\delta_{i,j+1}}
\end{equation}

\begin{equation}
\label{quantgl3}
[\fe_i,\ff_j]=\delta_{i,j}\frac{\fk_i-\fk_i^{-1}}{v-v^{-1}},\
\fk_i=\ft_i\ft_{i+1}^{-1}
\end{equation}

\begin{equation}
\label{quantgl4} [e_i,e_j]=[f_i,f_j]=0\ (|i-j|>1)
\end{equation}

\begin{equation}
\label{quantgl5} [\fe_i,[\fe_i,\fe_{i\pm
1}]_v]_v=[\ff_i,[\ff_i,\ff_{i\pm 1}]_v]_v=0,\ [a,b]_v:=ab-vba
\end{equation}

 The subalgebra
generated by $\{\fk_i,\fk_i^{-1},\fe_i,\ff_i\}_{1\leq i\leq n-1}$ is
isomorphic to $U_v(\fsl_n)$. We denote by $U_v(\fgl_n)_{\leq0}$ the
subalgebra of $U_v(\fgl_n)$ generated by $\ft_i,\ft_i^{-1},\ff_i$.
It acts on the field $\BC(\widetilde{T}\times\BC^*)$ as follows:
$\ff_i$ acts trivially for any $1\leq i\leq n-1$, and $\ft_i$ acts
by multiplication by $t_iv^{i-1}$. We define the {\em universal
Verma module} $\fM$ over $U_v(\fgl_n)$ as
$\fM:=U_v(\fgl_n)\otimes_{U_v(\fgl_n)_{\leq0}}\BC(\widetilde{T}\times\BC^*)$.

We define the following operators on $M$:
\begin{equation}
\label{quantgl1} \ft_i=t_iv^{d_{i-1}-d_i+i-1}:\ M_{\ul{d}}\to
M_{\ul{d}}
\end{equation}
\begin{equation}
\label{quantgl2} \fe_i=t_{i+1}^{-1}v^{d_{i+1}-d_i-i+1}\bp_*\bq^*:\
M_{\ul{d}}\to M_{\ul{d}-i}
\end{equation}
\begin{equation}
\label{quantgl3}
\ff_i=-t_i^{-1}v^{d_i-d_{i-1}+i}\bq_*(L_i\otimes\bp^*):\
M_{\ul{d}}\to M_{\ul{d}+i}
\end{equation}

The following result is Theorem~2.12 of~\cite{bf}.
\begin{thm}
\label{brav} These operators satisfy the relations in $U_v(\fgl_n)$,
i.e. they give rise to the action of $U_v(\fgl_n)$ on $M$. Moreover,
there is a unique isomorphism $\Psi:\ M\to\fM$ carrying
$[\CO_{\fQ_0}]\in M$ to the lowest weight vector
$1\in\BC(\widetilde{T}\times\BC^*)\subset\fM$.
\end{thm}

\begin{rem}
 These notations coincide with those from~\cite{bf} (see Theorem 2.12 and Conjecture 3.7 of {\em
loc. cit.}) after the Chevalley involution and a slight
renormalization (which makes formulas slightly shorter).
\end{rem}

\subsection{Gelfand-Tsetlin basis of the universal Verma module}
The construction of the Gelfand-Tsetlin basis for the
representations of quantum $\fgl_n$ goes back to M.~Jimbo~\cite{j}.
We will follow the approach of~\cite{mtz}. To a collection
$\widetilde{\ul{d}}=(d_{ij}),\ n-1\geq i\geq j$, we associate a {\em
Gelfand-Tsetlin pattern}
$\Lambda=\Lambda(\widetilde{\ul{d}}):=(\lambda_{ij}),\ n\geq i\geq
j$, as follows: $v^{\lambda_{nj}}:=t_jv^{j-1},\ n\geq j\geq 1;\
v^{\lambda_{ij}}:=t_jv^{j-1-d_{ij}},\ n-1\geq i\geq j\geq1$. Now we
define $\xi_{\widetilde{\ul{d}}}=\xi_\Lambda\in\fM$ by the
formula~(5.12) of~\cite{mtz}. According to Proposition~5.1 of {\em
loc. cit.}, the set $\{\xi_{\widetilde{\ul{d}}}\}$ (over all
collections $\widetilde{\ul{d}}$) forms a basis of $\fM$.

 According to the Thomason localization theorem, restriction to the
$\widetilde{T}\times\BC^*$-fixed point set induces an isomorphism
$$K^{\widetilde{T}\times\BC^*}(\fQ_{\ul{d}})
\otimes_{K^{\widetilde{T}\times\BC^*}(pt)}
\on{Frac}(K^{\widetilde{T}\times\BC^*}(pt))\iso
K^{\widetilde{T}\times\BC^*}(\fQ_{\ul{d}}^{\widetilde{T}\times\BC^*})
\otimes_{K^{\widetilde{T}\times\BC^*}(pt)}
\on{Frac}(K^{\widetilde{T}\times\BC^*}(pt))$$

The structure sheaves $[\widetilde{\ul{d}}]$ of the
$\widetilde{T}\times\BC^*$-fixed points $\widetilde{\ul{d}}$ (see
~\ref{fixed points}) form a basis in
$\bigoplus_{\ul{d}}K^{\widetilde{T}\times\BC^*}
(\fQ_{\ul{d}}^{\widetilde{T}\times\BC^*})
\otimes_{K^{\widetilde{T}\times\BC^*}(pt)}\on{Frac}
(K^{\widetilde{T}\times\BC^*}(pt))$. The embedding of a point
$\widetilde{\ul{d}}$ into $\fQ_{\ul{d}}$ is a proper morphism, so
the direct image in the equivariant $K$-theory is well defined, and
we will denote by $\{[\widetilde{\ul{d}}]\}\in M_{\ul{d}}$ the
direct image of the structure sheaves of the point
$\widetilde{\ul{d}}$. The set $\{[\widetilde{\ul{d}}]\}$ forms a
basis of $M$.

The following result is Theorem~6.3 of~\cite{fr} and Corollary 2.20
of~\cite{bf}.

\begin{thm}
\label{feigin} a) Isomorphism $\Psi:\ M\iso\fM$ of
Theorem~\ref{brav} takes $\{[\widetilde{\ul{d}}]\}$ to
$$(v^2-1)^{-|\ul{d}|}\prod_{
j}{t_j^{\sum_{i\geq
j}{d_{i,j}}}}v^{\sum_{i}{id_i}-\frac{|\ul{d}|}{2}-\frac{\sum_{i,j}{d_{i.j}^2}}{2}}\xi_{\widetilde{\ul{d}}}.$$

b) Matrix coefficients of the operators $\fe_i, \ff_i$ in the fixed
point basis $\{[\widetilde{\ul{d}}]\}$ of $M$ are as follows:
$$\ff_{i[\widetilde{\ul{d}},\widetilde{\ul{d}}{}']}=
-t_i^{-1}v^{d_i-d_{i-1}+i} t_j^2v^{-2d_{i,j}} \times$$
$$(1-v^2)^{-1}\prod_{j\ne k\leq i}(1-t_j^2t_k^{-2}v^{2d_{i,k}-2d_{i,j}})^{-1}
\prod_{k\leq i-1}(1-t_j^2t_k^{-2}v^{2d_{i-1,k}-2d_{i,j}})$$ if
$\widetilde{\ul{d}}{}'=\widetilde{\ul{d}}+\delta_{i,j}$ for certain
$j\leq i$;
$$\fe_{i[\widetilde{\ul{d}},\widetilde{\ul{d}}{}']}=
t_{i+1}^{-1}v^{d_{i+1}-d_i+1-i}\times$$
$$(1-v^2)^{-1}\prod_{j\ne k\leq i}(1-t_k^2t_j^{-2}v^{2d_{i,j}-2d_{i,k}})^{-1}
\prod_{k\leq i+1}(1-t_k^2t_j^{-2}v^{2d_{i,j}-2d_{i+1,k}})$$ if
$\widetilde{\ul{d}}{}'=\widetilde{\ul{d}}-\delta_{i,j}$ for certain
$j\leq i$.

All the other matrix coefficients of $\fe_i,\ff_i$ vanish.
\end{thm}

\subsection{Quantum loop algebra $\textbf{U}_v(\textbf{L}{\fsl_n})$}
\label{yang fin} Let $(a_{kl})_{1\leq k,l\leq n-1}=A_{n-1}$ stand
for the Cartan matrix of $\fsl_n$. For the quantum loop algebra
$\textbf{U}_v(\textbf{L}{\fsl_n})$ we follow the notations
of~\cite{nak-quiver}. Namely, the quantum loop algebra
$U_v(\textbf{L}{\fsl_n})$ is an associative algebra over $\mathbb
Q(v)$ generated by $e_{k,r}$, $f_{k,r}$, $v^{\pm h_k}$, $h_{k,m}$
$(1\leq k,l\leq n-1, r \in \BZ,  m\in \BZ\setminus \{0\})$ with the
following defining relations:

\begin{equation}
\label{1} \psi_k^{s}(z)\psi_l^{s'}(w)=\psi_l^{s'}(w)\psi_k^{s}(z)
\end{equation}

\begin{equation}
\label{2} (z-v^{\pm a_{kl}}w)\psi_l^s(z)x_k^{\pm}(w)=x_k^{\pm}(w)
\psi_l^s(z)(v^{\pm a_{kl}}z-w)
\end{equation}

\begin{equation}
\label{3} [x_k^{+}(z), x_l^{-}(w)]=\frac{\delta_{kl}}{v-v^{-1}}
\{\delta(w/z)\psi_k^{+}(w)-\delta(z/w)\psi_k^{-}(z) \}
\end{equation}

\begin{equation}
\label{4} (z-v^{\pm
2}w)x_k^{\pm}(z)x_k^{\pm}(w)=x_k^{\pm}(w)x_k^{\pm}(z)(v^{\pm 2}z-w)
\end{equation}

\begin{equation}
\label{5} (z-v^{\pm
 a_{k,l}}w)x_k^{\pm}(z)x_l^{\pm}(w)=x_l^{\pm}(w)x_k^{\pm}(z)(v^{\pm
 a_{k,l}}z-w),\ k\neq l
\end{equation}

\begin{equation}
\label{6} \{x_i^{s}(z_1)x_i^{s}(z_2)x_{i\pm
1}^{s}(w)-(v+v^{-1})x_i^{s}(z_1)x_{i\pm
1}^{s}(w)x_i^{s}(z_2)+x_{i\pm
1}^{s}(w)x_i^{s}(z_1)x_i^{s}(z_2)\}+\{z_1\longleftrightarrow z_2
\}=0
\end{equation}

where $s,s'=\pm$. Here $\delta (z), x_k^{\pm}(z), \psi_k^{\pm}(z)$
are generating functions defined as following

$$\delta(z):=\sum_{r=-\infty}^\infty z^r,\
x_k^{+}(z):=\sum_{r=-\infty}^\infty e_{k,r}z^{-r},\
x_k^{-}(z):=\sum_{r=-\infty}^\infty f_{k,r}z^{-r},$$
$$\psi_k^{\pm}(z):=v^{\pm h_k} \exp \left(\pm(v-v^{-1})\sum_{m=1}^\infty
h_{k,\pm m}z^{\mp m}\right).$$

\subsection{Action of $\textbf{U}_v(\textbf{L}{\fsl_n})$ on $M$}
\label{yang cor} For any $0\leq i\leq n$ we will denote by
$\ul{\CW}{}_i$ the tautological $i$-dimensional vector bundle on
$\fQ_{\ul{d}}\times\bC$. Let
$\pi:\fQ_{\ul{d}}\times(\bC\backslash\{\infty\})\rightarrow
\fQ_{\ul{d}}$ denote the standard projection. We define the
generating series $\bb_i(z)$ with coefficients in the equivariant
$K$-theory of $\fQ_{\ul{d}}$ as follows:
$$\bb_i(z):=\Lambda^{\bullet}_{-1/z}(\pi_*(\ul{\CW}{}_i\mid_{\bC\backslash\{\infty\}}))=
1+\sum_{j\geq
1}{\Lambda^j(\pi_*(\ul{\CW}{}_i\mid_{\bC\backslash\{\infty\}}))(-z^{-1})^j}:\
M_{\ul{d}}\to M_{\ul{d}}[[z^{-1}]]$$

 Let $v$ stand for the character of $\widetilde{T}\times\BC^*:\
(\ul{t},v)\mapsto v$. We define the line bundle $L'_k:=v^kL_k$ on
the correspondence $\sE_{\ul{d},k}$, that is $L'_k$ and $L_k$ are
isomorphic as line bundles but the equivariant structure of $L'_k$
is obtained from the equivariant structure of $L_k$ by the twist by
a character $v^k$.

We also define the operators
\begin{equation}
\label{ddvas}
e_{k,r}:=t_{k+1}^{-1}v^{d_{k+1}-d_k+1-k}\bp_*((L'_k)^{\otimes
r}\otimes \bq^*):\ M_{\ul{d}}\to M_{\ul{d}-k}
\end{equation}

\begin{equation}
\label{ttris} f_{k,r}:=-t_k^{-1}v^{d_k-d_{k-1}+k}\bq_*(L_k\otimes
(L'_k)^{\otimes r}\otimes \bp^*):\ M_{\ul{d}}\to M_{\ul{d}+k}
\end{equation}

 Consider the following generating series of operators on $M$:

\begin{equation}
\label{dvas}
x_{k}^+(z)=\sum_{r=-\infty}^\infty e_{k,r}z^{-r}:\
M_{\ul{d}}\to M_{\ul{d}-k}[[z,z^{-1}]]
\end{equation}

\begin{equation}
\label{tris}
x_{k}^{-}(z)=\sum_{r=-\infty}^\infty f_{k,r}z^{-r}:\
M_{\ul{d}}\to M_{\ul{d}+k}[[z,z^{-1}]]
\end{equation}

\begin{multline}
\label{raz1} \psi_k^{\pm}(z)=\sum_{r=0}^{\pm\infty}
\psi^{\pm}_{k,r}z^{-r}:=t_{k+1}^{-1}t_kv^{d_{k+1}-2d_k+d_{k-1}-1}\times
\\ \left(\bb_k(zv^{-k-2})^{-1}\bb_k(zv^{-k})^{-1}\bb_{k-1}(zv^{-k})\bb_{k+1}(zv^{-k-2})\right)^{\pm}:M_{\ul{d}}\to  M_{\ul{d}}[[z^{\mp 1}]]
\end{multline}


 where $\left(\ \right)^{\pm}$ denotes the expansion at $z=\infty,\ 0$, respectively.

\begin{thm}
\label{var} These generating series of operators $\psi_{k}^\pm(z),
x_{k}^\pm(z)$ on $M$ satisfy the relations in
$\textbf{U}_v(\textbf{L}{\fsl_n})$, i.e. they give rise to the
action of $\textbf{U}_v(\textbf{L}{\fsl_n})$ on $M$.
\end{thm}

\begin{rem} {\em
For the quantum group $U_v(\fsl_n)$ (generated by $e_{k,0},
f_{k,0},\psi_{k,0}^{\pm}$ in $U_v({\textbf{L}\fsl_n})$) we get
formulas~(\ref{quantgl1}--\ref{quantgl3}).
Formulas~(\ref{dvas}--\ref{raz1}) are very similar to those for
equivariant cohomology in~\cite{ffnr}.}
\end{rem}

\begin{defn}
 To each $\widetilde{\ul{d}}$ we assign a collection of $\widetilde{T}\times\BC^*$-weights $s_{i,j}:=t_j^2v^{-2d_{ij}}$.
\end{defn}

\begin{prop}
\label{matrix_elements} a) The matrix coefficients of the operators
$f_{i,r}, e_{i,r}$ in the fixed point basis
$\{[\widetilde{\ul{d}}]\}$ of $M$ are as follows:
$$f_{{i,r}[\widetilde{\ul{d}},\widetilde{\ul{d}}{}']}=
-t_i^{-1}v^{d_i-d_{i-1}+i} s_{i,j}(s_{i,j}v^i)^r
(1-v^2)^{-1}\prod_{j\ne k\leq i}(1-s_{i,j}s_{i,k}^{-1})^{-1}
\prod_{k\leq i-1}(1-s_{i,j}s_{i-1,k}^{-1})$$ if
$\widetilde{\ul{d'}}=\widetilde{\ul{d}}+\delta_{i,j}$ for certain
$j\leq i$;
$$e_{{i,r}[\widetilde{\ul{d}},\widetilde{\ul{d}}{}']}=
t_{i+1}^{-1}v^{d_{i+1}-d_i+1-i}(
s_{i,j}v^{i+2})^r(1-v^2)^{-1}\prod_{j\ne k\leq
i}(1-s_{i,k}s_{i,j}^{-1})^{-1} \prod_{k\leq
i+1}(1-s_{i+1,k}s_{i,j}^{-1})$$ if
$\widetilde{\ul{d'}}=\widetilde{\ul{d}}-\delta_{i,j}$ for certain
$j\leq i$.

All the other matrix coefficients of $e_{i,r},f_{i,r}$ vanish.

\noindent b) The eigenvalue of $\psi_i^{\pm}(z)$ on
$\{[\widetilde{\ul{d}}]\}$ equals

$$ t_{i+1}^{-1}t_iv^{d_{i+1}-2d_{i}+d_{i-1}-1}\prod_{j\le
i}(1-z^{-1}v^{i+2}s_{i,j})^{-1}(1-z^{-1}v^is_{i,j})^{-1}\prod_{j\le
i+1}(1-z^{-1}v^{i+2}s_{i+1,j})\prod_{j\le
i-1}(1-z^{-1}v^is_{i-1,j}),
$$
where it is expanded in $z^{\mp 1}$  depending on the sign $\pm$.

\end{prop}

\begin{proof}

 a) Follows directly from Theorem~\ref{feigin}b).

 b) Follows from the multiplicativity of $\Lambda^{\bullet}_z(L)$ on long exact sequences of coherent sheaves and the fact
that $\{s_{i,j}\}_{j\leq i}$ is the set of
$\widetilde{T}\times\BC^*$-characters in the stalk of
$\pi_*(\ul{\CW}{}_i\mid_{\bC\backslash\{\infty\}})$ at the fixed
point $\{[\widetilde{\ul{d}}]\}\in \fQ_{\ul{d}}$.
\end{proof}

\label{quotient}
 Now we formulate a corollary which will be used in Section~\ref{psay}.
For any $0\leq m<i\leq n$ we will denote by $\ul{\CW}{}_{mi}$ the
quotient $\ul{\CW}{}_i/\ul{\CW}{}_m$ of the tautological vector
bundles on $\fQ_{\ul{d}}\times\bC$. Similarly to the above, we
introduce the generating series:
$$\bb_{mi}(z):=\Lambda^{\bullet}_{-1/z}(\pi_*(\ul{\CW}{}_{mi}\mid_{\bC\backslash \{\infty\}})):\ M_{\ul{d}}\to M_{\ul{d}}[[z^{-1}]]$$

\begin{cor}\label{faktor}
For any $m<i$ we have
$$\psi_i^{\pm}(z)\mid
_{M_{\ul{d}}}=t_{i+1}^{-1}t_iv^{d_{i+1}-2d_i+d_{i-1}-1}
\left(\bb_{mi}(zv^{-i-2})^{-1}\bb_{mi}(zv^{-i})^{-1}\bb_{m,i-1}(zv^{-i})\bb_{m,i+1}(zv^{-i-2})\right)^{\pm}.$$

\end{cor}

\begin{proof}  Since $\Lambda^{\bullet}_z(L):=\sum_{j\geq 0}
z^i \Lambda^iL$ is multiplicative on long exact sequences, we have:
$$\Lambda^{\bullet}_{-1/z}(\ul{\CW}{}_i)=\Lambda^{\bullet}_{-1/z}(\ul{\CW}{}_m)\Lambda^{\bullet}_{-1/z}(\ul{\CW}{}_{mi}),$$
while on the other hand
$$\Lambda^{\bullet}_{-1/z}(\ul{\CW}{}_i)=\bb_i(z),\ \Lambda^{\bullet}_{-1/z}(\ul{\CW}{}_{mi})=\bb_{mi}(z),\ \Lambda^{\bullet}_{-1/z}(\ul{\CW}{}_m)=\bb_m(z).$$
 Now the result follows from~(\ref{raz1}).
\end{proof}

\section{Proof of Theorem~\ref{var}}
\label{techn}

Let us check equation~(\ref{4}) firstly. We will prove it for
$x_k^{-}$ (case $x_k^{+}$ is entirely analogous).

\begin{proof}
  We need to verify
$f_{i,a+1}f_{i,b}-v^{-2}f_{i,a}f_{i,b+1}=v^{-2}f_{i,b}f_{i,a+1}-f_{i,b+1}f_{i,a}$
for any integers $a, b$. Let us compute both sides in the fixed
point basis:

\medskip
\noindent a)
$[\widetilde{\ul{d}},\widetilde{\ul{d}}{}'=\widetilde{\ul{d}}+\delta_{i,j_1}+\delta_{i,j_2}]$
($j_1 \neq j_2$).
$$\left(f_{i,a+1}f_{i,b}-v^{-2}f_{i,a}f_{i,b+1}\right)_{[\widetilde{\ul{d}},\widetilde{\ul{d}}{}']}=
Pv^{i(a+b+1)}\times$$
$$\left[v^2s_{i,j_1}^bs_{i,j_2}^{a+1}(1-s_{i,j_1}s_{i,j_2}^{-1})^{-1}(1-v^2s_{i,j_2}s_{i,j_1}^{-1})^{-1}-
s_{i,j_1}^{b+1}s_{i,j_2}^{a}(1-s_{i,j_1}s_{i,j_2}^{-1})^{-1}(1-v^2s_{i,j_2}s_{i,j_1}^{-1})^{-1}+\{j_1\longleftrightarrow
j_2\}\right]$$
$$=Pv^{i(a+b+1)}\left[(v^2s_{i,j_1}^bs_{i,j_2}^{a+1}-s_{i,j_1}^{b+1}s_{i,j_2}^a)(1-s_{i,j_1}s_{i,j_2}^{-1})^{-1}(1-v^2s_{i,j_2}s_{i,j_1}^{-1})^{-1}+
\{j_1\longleftrightarrow j_2\}\right].$$

\medskip

Similarly
$$\left(v^{-2}f_{i,b}f_{i,a+1}-f_{i,b+1}f_{i,a}\right)_{[\widetilde{\ul{d}},\widetilde{\ul{d}}{}']}=
Pv^{i(a+b+1)}\times$$
$$\left[(s_{i,j_1}^{a+1}s_{i,j_2}^b-v^2s_{i,j_1}^as_{i,j_2}^{b+1})(1-s_{i,j_1}s_{i,j_2}^{-1})^{-1}(1-v^2s_{i,j_2}s_{i,j_1}^{-1})^{-1}+
\{j_1\longleftrightarrow j_2\}\right],
$$

\medskip

where
$$P=t_i^{-2}v^{2d_i-2d_{i-1}+2i-1}s_{i,j_1}s_{i,j_2} \times$$
$$(1-v^2)^{-2}\prod_{j_1,j_2\ne k\leq i
}(1-s_{i,j_1}s_{i,k}^{-1})^{-1}(1-s_{i,j_2}s_{i,k}^{-1})^{-1}
\prod_{k\leq
i-1}(1-s_{i,j_1}s_{i-1,k}^{-1})(1-s_{i,j_2}s_{i-1,k}^{-1}).
$$

\medskip

So we have to prove that
$$(s_{i,j_1}^bs_{i,j_2}^{a+1}-v^{-2}s_{i,j_1}^{b+1}s_{i,j_2}^a-v^{-2}s_{i,j_1}^{a+1}s_{i,j_2}^b+s_{i,j_1}^as_{i,j_2}^{b+1})(1-s_{i,j_1}s_{i,j_2}^{-1})^{-1}(1-v^2s_{i,j_2}s_{i,j_1}^{-1})^{-1}=$$
$$(s_{i,j_1}^bs_{i,j_2}^a(s_{i,j_2}-v^{-2}s_{i,j_1})+s_{i,j_1}^as_{i,j_2}^b(s_{i,j_2}-v^{-2}s_{i,j_1}))s_{i,j_1}s_{i,j_2}(s_{i,j_2}-s_{i,j_1})^{-1}(s_{i,j_1}-v^2s_{i,j_2})^{-1}=$$
$$=\frac{s_{i,j_1}s_{i,j_2}(s_{i,j_1}^as_{i,j_2}^b+s_{i,j_1}^bs_{i,j_2}^a)}{v^2(s_{i,j_1}-s_{i,j_2})}$$
is antisymmetric with respect to $\{j_1\longleftrightarrow j_2\}$
which is obvious.

\medskip
\noindent b)
$[\widetilde{\ul{d}},\widetilde{\ul{d}}{}'=\widetilde{\ul{d}}+2\delta_{i,j_1}]$.

In this case define
$$P':=t_i^{-2}v^{2d_i-2d_{i-1}+2i-1}s_{i,j_1}^2
\times$$ $$(1-v^2)^{-2}\prod_{j_1\ne k\leq i
}(1-s_{i,j_1}s_{i,k}^{-1})^{-1}(1-v^{-2}s_{i,j_1}s_{i,k}^{-1})^{-1}
\prod_{k\leq
i-1}(1-s_{i,j_1}s_{i-1,k}^{-1})(1-v^{-2}s_{i,j_1}s_{i-1,k}^{-1}).
$$

Then:
$$\left(f_{i,a+1}f_{i,b}-v^{-2}f_{i,a}f_{i,b+1}\right)_{[\widetilde{\ul{d}},\widetilde{\ul{d}}{}']}=
P'v^{i(a+b+1)}s_{i,j_1}^{a+b+1}(v^{-2(a+1)}-v^{-2}v^{-2a})=0=\left(v^{-2}f_{i,b}f_{i,a+1}-f_{i,b+1}f_{i,a}\right)_{[\widetilde{\ul{d}},\widetilde{\ul{d}}{}']}.$$

\medskip

So the equality holds again.
\end{proof}

\medskip

Let us check~(\ref{5}) now. We will prove it only for $x_k^{-}$
again.

\begin{proof}
  If $|k-l|>1$ then it is obvious that in the fixed point basis
the formulas are the same. So let us check it for $l=i+1, k=i$. In
other words, for any integers $a, b$ we have to verify
$f_{i,a+1}f_{i+1,b}-vf_{i,a}f_{i+1,b+1}=vf_{i+1,b}f_{i,a+1}-f_{i+1,b+1}f_{i,a}$.

 Let us compute matrix coefficients corresponding to the pair $[\widetilde{\ul{d}},\widetilde{\ul{d}}{}'=\widetilde{\ul{d}}+\delta_{i,j_1}+\delta_{i+1,j_2}]$
for both sides (here $j_1$ and $j_2$ might be equal).

We have
$$(f_{i,a+1}f_{i+1,b}-vf_{i,a}f_{i+1,b+1})_{[\widetilde{\ul{d}},\widetilde{\ul{d}}{}']}=
Pv\left(1-s_{i+1,j_2}s_{i,j_1}^{-1}\right)
\left[v^{i(a+1)+(i+1)b}s_{i+1,j_2}^bs_{i,j_1}^{a+1}-v^{ia+(i+1)(b+1)+1}s_{i+1,j_2}^{b+1}s_{i,j_1}^a\right],
$$

$$(vf_{i+1,b}f_{i,a+1}-f_{i+1,b+1}f_{i,a})_{[\widetilde{\ul{d}},\widetilde{\ul{d}}{}']}=
P\left(1-v^2s_{i+1,j_2}s_{i,j_1}^{-1}\right)
\left[v^{i(a+1)+(i+1)b+1}s_{i+1,j_2}^bs_{i,j_1}^{a+1}-v^{ia+(i+1)(b+1)}s_{i+1,j_2}^{b+1}s_{i,j_1}^a\right],
$$

where
$$P=t_{i+1}^{-1}t_i^{-1}v^{d_{i+1}-d_{i-1}+2i}(1-v^2)^{-2}s_{i,j_1}s_{i+1,j_2} \times$$
$$\prod_{j_1\ne k\leq i
}(1-s_{i,j_1}s_{i,k}^{-1})^{-1}\prod_{k\leq
i-1}(1-s_{i,j_1}s_{i-1,k}^{-1})\times \prod_{j_2\ne k\leq i+1
}(1-s_{i+1,j_2}s_{i+1,k}^{-1})^{-1}\prod_{j_1\ne k\leq
i}(1-s_{i+1,j_2}s_{i,k}^{-1}).$$

\medskip

After dividing both right hand sides by
$Ps_{i,j_1}^{a-1}s_{i+1,j_2}^bv^{ia+(i+1)b}$ we get an equality:

$v(v^is_{i,j_1}-v^{i+2}s_{i+1,j_2})(s_{i,j_1}-s_{i+1,j_2})=(v^{i+1}s_{i,j_1}-v^{i+1}s_{i+1,j_2})(s_{i,j_1}-v^2s_{i+1,j_2})$.
\end{proof}

\medskip

Let us check~(\ref{3}) for the case $k \neq l$.
\begin{proof}
We have to prove $e_{k,a}f_{l,b}=f_{l,b}e_{k,a}$ for any integers
$a, b$.

 This is obvious when $|k-l|>1$, since matrix coefficients in the fixed
point basis are the same. Let us check the only nontrivial case:
$k=i, l=i+1$ (pair $(k=i+1, l=i)$ is analogous). We consider the
pair of fixed points
$[\widetilde{\ul{d}},\widetilde{\ul{d}}{}'=\widetilde{\ul{d}}-\delta_{i,j_1}+\delta_{i+1,j_2}]$
(here $j_1, j_2$ might be equal).

$$e_{i,a}f_{i+1,b}\mid_{[\widetilde{\ul{d}},\widetilde{\ul{d}}{}']}=P(1-s_{i+1,j_2}s_{i,j_1}^{-1})(1-v^{-2}s_{i+1,j_2}s_{i,j_1}^{-1}),$$

$$f_{i+1,b}e_{i,a}\mid_{[\widetilde{\ul{d}},\widetilde{\ul{d}}{}']}=P(1-v^{-2}s_{i+1,j_2}s_{i,j_1}^{-1})(1-s_{i+1,j_2}s_{i,j_1}^{-1}),$$

where
$$P=-t_{i+1}^{-1}t_i^{-1}v^{2d_{i+1}-2d_i+4}(1-v^2)^{-2}s_{i,j_1}^as_{i+1,j_2}^bv^{a(i+2)+b(i+1)}\times$$
$$\prod_{j_2\ne k\leq i+1
}(1-s_{i+1,j_2}s_{i+1,k}^{-1})^{-1}\prod_{j_1\ne k\leq
i}(1-s_{i+1,j_2}s_{i,k}^{-1})\times \prod_{j_1\ne k\leq i
}(1-s_{i,j_1}^{-1}s_{i,k})^{-1}\prod_{j_2\ne k\leq
i+1}(1-s_{i,j_1}^{-1}s_{i+1,k}).$$

This completes the proof of equation~(\ref{3}) in the case $k\ne l$.
\end{proof}

\medskip

Let us check~(\ref{6}) fourthly. We will prove it only for $x_k^{-}$
again.

\begin{proof}
We have to prove that for any integers $a, b, c$ and $j=i\pm 1$ the
following equality holds:
$$\{f_{i,a}f_{i,b}f_{j,c
}-(v+v^{-1})f_{i,a}f_{j,c}f_{i,b}+f_{j,c}f_{i,a}f_{i,b}\}+\{a\longleftrightarrow
b\}=0.$$

Let us consider the case $j=i+1$ (the second case is similar). We
will show that matrix coefficients in the fixed point basis of the
first bracket are antisymmetric with respect to a change
$\{a\longleftrightarrow b\}$.

\medskip
\noindent a)
$[\widetilde{\ul{d}},\widetilde{\ul{d}}{}'=\widetilde{\ul{d}}+\delta_{i,j_1}+\delta_{i,j_2}+\delta_{i+1,j_3}]$
($j_1 \neq j_2$).
$$
f_{i,a}f_{i,b}f_{i+1,c}\mid_{[\widetilde{\ul{d}},\widetilde{\ul{d}}{}']}=$$$$Pv^2\left[s_{i,j_1}^as_{i,j_2}^b(1-s_{i+1,j_3}s_{i,j_2}^{-1})(1-s_{i+1,j_3}s_{i,j_1}^{-1})
(1-s_{i,j_2}s_{i,j_1}^{-1})^{-1}(1-v^2s_{i,j_1}s_{i,j_2}^{-1})^{-1}+\{j_1\longleftrightarrow
j_2\}
 \right],
$$

$$
f_{i,a}f_{i+1,c}f_{i,b}\mid_{[\widetilde{\ul{d}},\widetilde{\ul{d}}{}']}=$$$$Pv\left[s_{i,j_1}^as_{i,j_2}^b(1-v^2s_{i+1,j_3}s_{i,j_2}^{-1})(1-s_{i+1,j_3}s_{i,j_1}^{-1})
(1-s_{i,j_2}s_{i,j_1}^{-1})^{-1}(1-v^2s_{i,j_1}s_{i,j_2}^{-1})^{-1}+\{j_1\longleftrightarrow
j_2\}
 \right],
$$

$$
f_{i+1,c}f_{i,a}f_{i,b}\mid_{[\widetilde{\ul{d}},\widetilde{\ul{d}}{}']}=$$$$P\left[s_{i,j_1}^as_{i,j_2}^b(1-v^2s_{i+1,j_3}s_{i,j_2}^{-1})(1-v^2s_{i+1,j_3}s_{i,j_1}^{-1})
(1-s_{i,j_2}s_{i,j_1}^{-1})^{-1}(1-v^2s_{i,j_1}s_{i,j_2}^{-1})^{-1}+\{j_1\longleftrightarrow
j_2\}
 \right].
$$

\medskip

Thus:
$$(f_{i,a}f_{i,b}f_{i+1,c
}-(v+v^{-1})f_{i,a}f_{i+1,c}f_{i,b}+f_{i+1,c}f_{i,a}f_{i,b})_{[\widetilde{\ul{d}},\widetilde{\ul{d}}{}']}=Ps_{i,j_1}^{a}s_{i,j_2}^{b}\times$$$$
\left(\frac
{v^2(s_{i,j_2}-s_{i+1,j_3})(s_{i,j_1}-s_{i+1,j_3})}{{(s_{i,j_1}-s_{i,j_2})(s_{i,j_2}-v^2s_{i,j_1})}}
-
\frac{(1+v^2)(s_{i,j_2}-v^2s_{i+1,j_3})(s_{i,j_1}-s_{i+1,j_3})}{(s_{i,j_1}-s_{i,j_2})(s_{i,j_2}-v^2s_{i,j_1})}+\right.$$
$$\left.\frac{(s_{i,j_2}-v^2s_{i+1,j_3})(s_{i,j_1}-v^2s_{i+1,j_3})}{(s_{i,j_1}-s_{i,j_2})(s_{i,j_2}-v^2s_{i,j_1})}\right)+\{j_1\longleftrightarrow
j_2\}=$$$$P(1-v^2)\frac{s_{i+1,j_3}s_{i,j_1}^{a}s_{i,j_2}^{b}}{s_{i,j_1}-s_{i,j_2}}+\{j_1\longleftrightarrow
j_2\}=Ps_{i+1,j_3}(1-v^2)\frac{s_{i,j_1}^{a}s_{i,j_2}^{b}-s_{i,j_1}^{b}s_{i,j_2}^{a}}{s_{i,j_1}-s_{i,j_2}},
$$

where
$$
P=-t_{i+1}^{-1}t_i^{-2}v^{d_{i+1}+d_i-2d_{i-1}+3i-1}s_{i,j_1}s_{i,j_2}s_{i+1,j_3}^{c+1}v^{c(i+1)+i(a+b)}(1-v^2)^{-3}$$
$$\times \prod_{k\leq
i-1}\left((1-s_{i,j_2}s_{i-1,k}^{-1})(1-s_{i,j_1}s_{i-1,k}^{-1})\right)\times$$$$\prod_{j_3\ne
k\leq i+1 }(1-s_{i+1,j_3}s_{i+1,k}^{-1})^{-1}\prod_{j_1, j_2 \ne
k\leq i}(1-s_{i+1,j_3}s_{i,k}^{-1})\prod_{j_1, j_2\ne k\leq i
}(1-s_{i,j_2}s_{i,k}^{-1})^{-1}\prod_{j_1,j_2\ne k\leq i
}(1-s_{i,j_1}s_{i,k}^{-1})^{-1}.
$$

\medskip

We see that $$f_{i,a}f_{i,b}f_{i+1,c
}-(v+v^{-1})f_{i,a}f_{i+1,c}f_{i,b}+f_{i+1,c}f_{i,a}f_{i,b}\mid_{[\widetilde{\ul{d}},\widetilde{\ul{d}}{}']}=
Ps_{i+1,j_3}(1-v^2)\frac{s_{i,j_1}^{a}s_{i,j_2}^{b}-s_{i,j_1}^{b}s_{i,j_2}^{a}}{s_{i,j_1}-s_{i,j_2}}$$
is antisymmetric with respect to $a\longleftrightarrow b.$

\medskip
\noindent
b)
$[\widetilde{\ul{d}},\widetilde{\ul{d}}{}'=\widetilde{\ul{d}}+2\delta_{i,j_1}+\delta_{i+1,j_3}].$

By the same calculation one gets:
$$(f_{i,a}f_{i,b}f_{i+1,c
}-(v+v^{-1})f_{i,a}f_{i+1,c}f_{i,b}+f_{i+1,c}f_{i,a}f_{i,b})_{[\widetilde{\ul{d}},\widetilde{\ul{d}}{}']}=$$
$$P'\left[v^2(1-s_{i+1,j_3}s_{i,j_1}^{-1})-(1+v^2)(1-v^2s_{i+1,j_3}s_{i,j_1}^{-1})+(1-v^4s_{i+1,j_3}s_{i,j_1}^{-1})\right]=0,$$

where
$$P'=-t_{i+1}^{-1}t_i^{-2}v^{d_{i+1}+d_i-2d_{i-1}+3i-1}s_{i,j_1}^{a+b+2}s_{i+1,j_3}^{c+1}v^{c(i+1)+i(a+b)}(1-v^2)^{-3}$$
$$\times \prod_{k\leq
i-1}\left((1-v^{-2}s_{i,j_1}s_{i-1,k}^{-1})(1-s_{i,j_1}s_{i-1,k}^{-1})\right)\times$$$$\prod_{j_3\ne
k\leq i+1 }(1-s_{i+1,j_3}s_{i+1,k}^{-1})^{-1}\prod_{j_1 \ne k\leq
i}\left((1-s_{i+1,j_3}s_{i,k}^{-1})^{-1}(1-v^{-2}s_{i,j_1}s_{i,k}^{-1})(1-s_{i,j_1}s_{i,k}^{-1})\right)^{-1}.
$$

\medskip
This completes the proof of~(\ref{6}).
\end{proof}

\medskip
 Now we will introduce the series of operators
$\varphi_k^{\pm}(z)_{\mid M_{\ul{d}}}=\sum_{r=0}^{\pm\infty}
\varphi^{\pm}_{k,r}{_{\mid M_{\ul{d}}}z^{-r}}$
diagonalizable in the fixed point
basis and satisfying the equation
\begin{equation}
\label{1'} [x_k^{+}(z), x_k^{-}(w)]=\frac{1}{v-v^{-1}}
\{\delta(w/z)\varphi_k^{+}(w)-\delta(z/w)\varphi_k^{-}(z) \}
\end{equation}

We will show that equality (\ref{1'}) determine $\varphi_k^{\pm}(z)$
uniquely up to a particular choice of $\varphi_{i,0}^{\pm}$ (the
latter ambiguity is easily resolved by the formulas of
Theorem~\ref{brav} as explained below).
Let us further omit $\mid_{M_{\ul{d}}}$ for brevity. Next we will
check
\begin{equation}
\label{2'}
\varphi_k^{s}(z)\varphi_l^{s'}(w)=\varphi_l^{s'}(w)\varphi_k^{s}(z)
\end{equation}
\begin{equation}
\label{3'} (z-v^{\pm
a_{kl}}w)\varphi_l^s(z)x_k^{\pm}(w)=x_k^{\pm}(w)
\varphi_l^s(z)(v^{\pm a_{kl}}z-w)
\end{equation}

Finally by showing that $\varphi_k^{\pm}(z)=\psi_k^{\pm}(z)$ we will
get~(\ref{1}--\ref{3}) from~(\ref{1'}--\ref{3'}).

\medskip
 From Proposition~\ref{matrix_elements} one gets that $(v-v^{-1})[x_i^{+}(z),
 x_i^{-}(w)]$ is diagonalizable in the fixed point basis and moreover
 its eigenvalue at $\{[\widetilde{\ul{d}}]\}$ equals to $$\sum_{a,b\in \mathbb Z}{z^{-a}w^{-b}\chi_{i,a+b}},$$
where
$$\chi_{i,c}=-t_{i+1}^{-1}t_i^{-1}v^{d_{i+1}-d_{i-1}-1}(v^2-1)^{-1}\times$$$$\sum_{j\leq i}{s_{ij}\left(\prod_{j\neq k \leq i} {(1-s_{i,j}s_{i,k}^{-1})}^{-1}
 \prod_{j\neq k \leq i}{(1-v^2s_{i,k}s_{i,j}^{-1})}^{-1}
\prod_{k\leq i-1}{(1-s_{i,j}s_{i-1,k}^{-1})}\prod_{k\leq
i+1}{(1-v^2s_{i+1,k}s_{i,j}^{-1})}(s_{i,j}v^i)^c-\right.}$$$${\left.
v^2\prod_{j\neq k \leq i} {(1-s_{i,j}^{-1}s_{i,k})}^{-1}\prod_{j\neq
k \leq i}{(1-v^2s_{i,k}^{-1}s_{i,j})}^{-1} \prod_{k\leq
i-1}{(1-v^2s_{i,j}s_{i-1,k}^{-1})} \prod_{k\leq
i+1}{(1-s_{i+1,k}s_{i,j}^{-1})}(s_{i,j}v^{i+2})^c\right)}.$$

So as we want an equality $(v-v^{-1})[x_i^{+}(z),
x_i^{-}(w)]=\delta\left(\frac{z}{w}\right)\varphi_i^{+}(w)-\delta\left(\frac{w}{z}\right)\varphi_i^{-}(z)=$$$
\sum_{a,b|a+b>0}{z^{-a}w^{-b}\varphi_{i,a+b}^{+}}-\sum_{a,b|a+b<0}{z^{-a}w^{-b}\varphi_{i,a+b}^{-}}+\sum_{a,b|a+b=0}{z^{-a}w^{-b}(\varphi_{i,0}^{+}-\varphi_{i,0}^{-})}$$
to hold, we determine $\varphi_{i,s>0}^{+},
\varphi_{i,s<0}^{-},\varphi_{i,s=0}^{+}-\varphi_{i,s=0}^{-} $
uniquely as they are equal to the corresponding $\chi_{i,s}$.
Recalling results of~\cite{bf} we see that equality for
$\varphi_{i,0}^{+}-\varphi_{i,0}^{-}=\chi_{i,0}$ has a particular
solution
$\varphi_{i,0}^{+}=t_it_{i+1}^{-1}v^{d_{i+1}-2d_i+d_{i-1}-1},\varphi_{i,0}^{-}=t_i^{-1}t_{i+1}v^{-d_{i+1}+2d_i-d_{i-1}+1}$.
This determines all coefficients of the series $\varphi_i^{\pm}(z)$
(this particular choice of $\varphi_{i,0}^{\pm}$ is crucial for a
verification of $\varphi_i^\pm(z)=\psi_i^\pm(z)$).

 Since all operators $\varphi_{i,s}^{\pm}$ are
diagonalizable in the fixed point basis~(\ref{2'}) holds
automatically. So let us check~(\ref{3'}), i.e.
$$(z-v^{s' a_{kl}}w)\varphi_l^s(z)x_k^{s'}(w)=x_k^{s'}(w)
\varphi_l^s(z)(v^{s' a_{kl}}z-w).$$

\begin{proof}
We will check it for $k=l, s=+, s'=-$ as all other cases are
analogous (the case $k\ne l$ follows directly from~(\ref{5}) and the
construction of $\varphi_i^{\pm}(z)$).

 Now we are computing the matrix coefficients of both sides in the fixed
point basis at the pair
$[\widetilde{\ul{d}},\widetilde{\ul{d}}{}'=\widetilde{\ul{d}}+\delta_{i,p}].$
 Let us point out that
$f_{i,b+1}\mid_{[\widetilde{\ul{d}},\widetilde{\ul{d}}{}']}=f_{i,b}\mid_{[\widetilde{\ul{d}},\widetilde{\ul{d}}{}']}\cdot
s_{i,p}v^i$.
 And as $\varphi_{i,s\geq 0}^{+}$ are diagonalizable in the fixed point
basis we just need to verify that for any $a\geq 0$ we have:
$$(\varphi_{i,a+1}^{+}-v^{-2}s_{i,p}v^i\varphi_{i,a}^{+})\mid_{\widetilde{\ul{d}}+\delta_{i,p}}=(v^{-2}\varphi_{i,a+1}^{+}-s_{i,p}v^i\varphi_{i,a}^{+})\mid_{\widetilde{\ul{d}}}.$$
a) \emph{Case $a>0.$} Here we use the notations of Proposition
2.21,~\cite{bf}. Namely, define:
$$q:=v^2, s_j:=s_{ij}=t_j^2v^{-2d_{i,j}},
p_k:=s_{i-1,k}=t_k^2v^{-2d_{i-1,k}},
r_k:=s_{i+1,k}=t_k^2v^{-2d_{i+1,k}}.$$
 Then
$$\varphi_{i,a}^{+}\mid_{\widetilde{\ul{d}}}=P\prod_{j\leq
i} {s_j} \prod_{k\leq i-1} {p_k^{-1}} \left(\sum_{j\leq
i}{s_j^{-2}\prod_{k\leq i+1}{(s_j-qr_k)}\prod_{k\leq
i-1}{(p_k-s_j)}\prod_{j\ne k\leq
i}{\left((s_j-qs_k)^{-1}(s_k-s_j)^{-1}\right)}s_j^a} \right.-$$
$$\left.q\sum_{j\leq i}{s_j^{-2}\prod_{k\leq i+1}{(s_j-r_k)}\prod_{k\leq i-1}{(p_k-qs_j)}\prod_{j\ne k\leq
i}{\left((s_j-s_k)^{-1}(s_k-qs_j)^{-1}\right)}(qs_j)^a}\right).$$
$$\varphi_{i,a}^{+}\mid_{\widetilde{\ul{d}}+\delta_{i,p}}=P\prod_{j\leq
i} {s_j} \prod_{k\leq i-1} {p_k^{-1}}q^{-1} \left(\sum_{p\ne j\leq
i}{s_j^{-2}\prod_{k\leq i+1}{(s_j-qr_k)}\prod_{k\leq
i-1}{(p_k-s_j)}}\times \right.$$
$$\prod_{j,p\ne k\leq
i}{\left((s_j-qs_k)^{-1}(s_k-s_j)^{-1}\right)}(s_j-s_p)^{-1}(q^{-1}s_p-s_j)^{-1}s_j^a
-$$
$$q\sum_{p\ne j\leq i}{s_j^{-2}\prod_{k\leq i+1}{(s_j-r_k)}\prod_{k\leq i-1}{(p_k-qs_j)}}\times$$
$$\prod_{j,p\ne k\leq
i}{\left((s_j-s_k)^{-1}(s_k-qs_j)^{-1}\right)}(s_j-q^{-1}s_p)^{-1}(q^{-1}s_p-qs_j)^{-1}(qs_j)^a+$$
$$
s_p^{-2}q^2\prod_{k\leq i+1}{(q^{-1}s_p-qr_k)}\prod_{k\leq
i-1}{(p_k-q^{-1}s_p)}\prod_{p\ne k\leq
i}{\left((q^{-1}s_p-qs_k)^{-1}(s_k-q^{-1}s_p)\right)}(q^{-1}s_p)^a-$$
$$\left.qs_p^{-2}q^2\prod_{k\leq i+1}{(q^{-1}s_p-r_k)}\prod_{k\leq
i-1}{(p_k-s_p)}\prod_{p\ne k\leq
i}{\left((q^{-1}s_p-s_k)^{-1}(s_k-s_p)^{-1}\right)s_p^a}\right),$$
where $P=-t_{i+1}^{-1}t_iv^{d_{i+1}-d_{i-1}-1+ia}(v^2-1)^{-1}$.

\medskip

Hence:
$$(v^{-2}\varphi_{i,a+1}^{+}-s_pv^i\varphi_{i,a}^{+})\mid_{\widetilde{\ul{d}}}=Pv^i\prod_{j\leq
i} {s_j} \prod_{k\leq i-1} {p_k^{-1}}
 \left(\sum_{p\ne j\leq
i}{s_j^{-2}\prod_{k\leq i+1}{(s_j-qr_k)}\prod_{k\leq
i-1}{(p_k-s_j)}}\times\right.$$
$$\prod_{j\ne k\leq
i}{\left((s_j-qs_k)^{-1}(s_k-s_j)^{-1}\right)}(q^{-1}s_j-s_p)s_j^a-$$
$$q\sum_{p\ne j\leq i}{s_j^{-2}\prod_{k\leq i+1}{(s_j-r_k)}\prod_{k\leq i-1}{(p_k-qs_j)}\prod_{j\ne k\leq
i}{\left((s_j-s_k)^{-1}(s_k-qs_j)^{-1}\right)}(s_j-s_p)(qs_j)^a}+$$
$$\left.s_p^{-2}\prod_{k\leq i+1}{(s_p-qr_k)}\prod_{k\leq i-1}{(p_k-s_p)}\prod_{p\ne k\leq
i}{\left((s_p-qs_k)^{-1}(s_k-s_p)^{-1}\right)}(q^{-1}-1)s_p^{a+1}\right).
$$

$$(\varphi_{i,a+1}^{+}-v^{-2}s_pv^i\varphi_{i,a}^{+})\mid_{\widetilde{\ul{d}}+\delta_{i,p}}=Pv^i\prod_{j\leq
i} {s_j} \prod_{k\leq i-1} {p_k^{-1}}q^{-1} \left(\sum_{p\ne j\leq
i}{s_j^{-2}\prod_{k\leq i+1}{(s_j-qr_k)}\prod_{k\leq
i-1}{(p_k-s_j)}}\times \right.$$
$$\prod_{j,p\ne k\leq
i}{\left((s_j-qs_k)^{-1}(s_k-s_j)^{-1}\right)}(s_j-s_p)^{-1}(q^{-1}s_p-s_j)^{-1}(s_j-q^{-1}s_p)s_j^a
-$$
$$q\sum_{p\ne j\leq i}{s_j^{-2}\prod_{k\leq i+1}{(s_j-r_k)}\prod_{k\leq i-1}{(p_k-qs_j)}}\times$$
$$\prod_{j,p\ne k\leq
i}{\left((s_j-s_k)^{-1}(s_k-qs_j)^{-1}\right)}(s_j-q^{-1}s_p)^{-1}(q^{-1}s_p-qs_j)^{-1}(qs_j-q^{-1}s_p)(qs_j)^a-$$
$$\left.qs_p^{-2}q^2\prod_{k\leq i+1}{(q^{-1}s_p-r_k)}\prod_{k\leq
i-1}{(p_k-s_p)}\prod_{p\ne k\leq
i}{\left((q^{-1}s_p-s_k)^{-1}(s_k-s_p)^{-1}\right)(1-q^{-1})s_p^{a+1}}\right).$$

It is straightforward to check that these two expressions coincide.

\medskip
\noindent
b) \emph{Case $a=0$.} In this case, the same argument as
used in a) shows
$$(\chi_{i,1}-v^{-2}s_{i,p}v^i\chi_{i,0})\mid_{\widetilde{\ul{d}}+\delta_{i,p}}=(v^{-2}\chi_{i,1}-s_{i,p}v^i\chi_{i,0})\mid_{\widetilde{\ul{d}}}.$$
Since $\varphi_{i,0}^{+}=\chi_{i,0}+\varphi_{i,0}^{-},\
\varphi_{i,1}^{+}=\chi_{i,1}$,
it suffices to verify
$v^{-2}\varphi_{i,0}^{-}\mid_{\widetilde{\ul{d}}+\delta_{i,p}}=\varphi_{i,0}^{-}\mid_{\widetilde{\ul{d}}}$,
which follows directly from the formula
$\varphi_{i,0}^{-}\mid_{\widetilde{\ul{d}}}=t_i^{-1}t_{i+1}v^{-d_{i+1}+2d_i-d_{i-1}+1}$.
\end{proof}

\medskip

  Finally, we rewrite formulas for $\varphi_i^{\pm}(z)$.
According to~(\ref{3'}), for any $a>0$ we have:
$$(\varphi_{l,a+1}^{+}-v^{-a_{k,l}}t_p^2v^{-2d_{k,p}}v^k\varphi_{l,a}^{+})\mid_{\widetilde{\ul{d}}+\delta_{k,p}}=
(v^{-a_{k,l}}\varphi_{l,a+1}^{+}-t_p^2v^{-2d_{k,p}}v^k\varphi_{l,a}^{+})\mid_{\widetilde{\ul{d}}},$$

i.e.
$$\varphi_l^{+}(z)(1-t_p^2v^{-a_{k,l}-2d_{k,p}+k}z^{-1})\mid_{\widetilde{\ul{d}}+\delta_{k,p}}=
\varphi_l^{+}(z)(v^{-a_{k,l}}-t_p^2v^{-2d_{k,p}+k}z^{-1})\mid_{\widetilde{\ul{d}}}.$$

 This is especially interesting whenever $a_{k,l}\ne 0$ providing the
following equalities:
\begin{equation}
\label{2.1}
\frac{\varphi_l^{+}(z)\mid_{\widetilde{\ul{d}}+\delta_{l+1,p}}}{\varphi_l^{+}(z)\mid_{\widetilde{\ul{d}}}}=
v\frac{1-z^{-1}v^{l}t_p^2v^{-2d_{l+1,p}}}{1-z^{-1}v^{l+2}t_p^2v^{-2d_{l+1,p}}}
\end{equation}

\begin{equation}
\label{2.2}
\frac{\varphi_l^{+}(z)\mid_{\widetilde{\ul{d}}+\delta_{l-1,p}}}{\varphi_l^{+}(z)\mid_{\widetilde{\ul{d}}}}=
v\frac{1-z^{-1}v^{l-2}t_p^2v^{-2d_{l-1,p}}}{1-z^{-1}v^{l}t_p^2v^{-2d_{l-1,p}}}
\end{equation}

\begin{equation}
\label{2.3}
\frac{\varphi_l^{+}(z)\mid_{\widetilde{\ul{d}}+\delta_{l,p}}}{\varphi_l^{+}(z)\mid_{\widetilde{\ul{d}}}}=
v^{-2}\frac{1-z^{-1}v^{l+2}t_p^2v^{-2d_{l,p}}}{1-z^{-1}v^{l-2}t_p^2v^{-2d_{l,p}}}
\end{equation}

Let $\widetilde{\ul{d_0}}=(d_{i,j}=0|\forall\ i,j)$, then recalling
the definition of $\varphi_i^+(z)$ we get
$$\varphi_i^{+}(z)\mid_{\widetilde{\ul{d_0}}}=t_{i+1}^{-1}t_iv^{-1}-t_{i+1}^{-1}t_i^{-1}v^{-1}(v^2-1)^{-1}t_i^2\times$$
$$\sum_{a\geq 1}{\prod_{k\leq i-1}{(1-t_i^2t_k^{-2})^{-1}}\prod_{k\leq i-1}{(1-v^2t_k^2t_i^{-2})^{-1}}\prod_{k\leq i-1}{(1-t_i^2t_k^{-2})}
\prod_{k\leq i+1}{(1-v^2t_k^2t_i^{-2})}(t_i^2v^iz^{-1})^a }=$$
$$t_{i+1}^{-1}t_iv^{-1}-t_{i+1}^{-1}t_iv^{-1}(v^2-1)^{-1}(1-v^2)(1-t_{i+1}^2t_i^{-2}v^2)\frac{t_i^2v^iz^{-1}}{1-t_i^2v^iz^{-1}}=
t_{i+1}^{-1}t_iv^{-1}(1-t_{i+1}^2v^{i+2}z^{-1})(1-t_i^2v^iz^{-1})^{-1}.$$

So
\begin{equation}
\label{2.4}
\varphi_i^{+}(z)\mid_{\widetilde{\ul{d_0}}}=t_{i+1}^{-1}t_iv^{-1}(1-t_{i+1}^2v^{i+2}z^{-1})(1-t_i^2v^iz^{-1})^{-1}.
\end{equation}

The following formula is a direct consequence
of~(\ref{2.1}--\ref{2.4}):
\begin{equation}
\label{2.5}
\varphi_i^{+}(z)=t_{i+1}^{-1}t_iv^{d_{i+1}-2d_i+d_{i-1}-1}\left(a_{i+1}(zv^{-i-2})a_{i-1}(zv^{-i})a_i(zv^{-i-2})^{-1}a_i(zv^{-i})^{-1}\right)^{+},
\end{equation}

\begin{equation}
\label{2.6} a_j(z)\mid_{\widetilde{\ul{d}}}:=\prod_{p\leq
j}{(1-z^{-1}t_p^2v^{-2d_{j,p}})}.
\end{equation}

\noindent
 Comparing~(\ref{2.5}--\ref{2.6}) to
Proposition~\ref{matrix_elements}b), we get
$\varphi_i^{+}(z)=\psi_i^{+}(z)$. Analogously:
$\varphi_i^{-}(z)=\psi_i^{-}(z)$. \textbf{\textit{Theorem~\ref{var}
is proved. }  }

\section{Parabolic sheaves and quantum toroidal algebra}
\label{psay} In this section we generalize our previous results to
the affine setting.

\subsection{Parabolic sheaves}
\label{PS} We recall the setup of section~3 of~\cite{bf}. Let $\bX$
be another smooth projective curve of genus zero. We fix a
coordinate $y$ on $\bX$, and consider the action of $\BC^*$ on $\bX$
such that $c(y)=c^{-2}y$. We have
$\bX^{\BC^*}=\{0_\bX,\infty_\bX\}$. Let $\bS$ denote the product
surface $\bC\times\bX$. Let $\bD_\infty$ denote the divisor
$\bC\times\infty_\bX\cup\infty_\bC\times\bX$. Let $\bD_0$ denote the
divisor $\bC\times0_\bX$.

Given an $n$-tuple of nonnegative integers
$\ul{d}=(d_0,\ldots,d_{n-1})$, a {\em parabolic sheaf $\CF_\bullet$
of degree $\ul{d}$} is an infinite flag of torsion free coherent
sheaves of rank $n$ on $\bS:\
\ldots\subset\CF_{-1}\subset\CF_0\subset\CF_1\subset\ldots\ $ such
that:

(a) $\CF_{k+n}=\CF_k(\bD_0)$ for any $k$;

(b) $ch_1(\CF_k)=k[\bD_0]$ for any $k$: the first Chern classes
are proportional to the fundamental class of $\bD_0$;

(c) $ch_2(\CF_k)=d_i$ for $i\equiv k\pmod{n}$;

(d) $\CF_0$ is locally free at $\bD_\infty$ and trivialized at
$\bD_\infty:\ \CF_0|_{\bD_\infty}=W\otimes\CO_{\bD_\infty}$;

(e) For $-n\leq k\leq0$ the sheaf $\CF_k$ is locally free at
$\bD_\infty$, and the quotient sheaves $\CF_k/\CF_{-n},\
\CF_0/\CF_k$ (both supported at $\bD_0=\bC\times0_\bX\subset\bS$)
are both locally free at the point $\infty_\bC\times0_\bX$;
moreover, the local sections of $\CF_k|_{\infty_\bC\times \bX}$ are
those sections of $\CF_0|_{\infty_\bC\times \bX}=W\otimes\CO_\bX$
which take value in $\langle w_1,\ldots,w_{n+k}\rangle\subset W$ at
$0_\bX\in \bX$.

\medskip

The fine moduli space $\CP_{\ul{d}}$ of degree $\ul{d}$ parabolic
sheaves exists and is a smooth connected quasiprojective variety of
dimension $2d_0+\cdots+2d_{n-1}$.

\subsection{Fixed points}
\label{fp} The group $\widetilde{T}\times\BC^*\times\BC^*$ acts
naturally on $\CP_{\ul{d}}$, and its fixed point set is finite. In
order to describe it, we recall the well known description of the
fixed point set of a $\BC^*\times\BC^*$-action on the Hilbert scheme
of points of $(\bC-\infty_\bC)\times(\bX-\infty_\bX)\cong \CC^2$.
The latter fixed points are parameterized by the Young diagrams, and
for a diagram $\lambda=(\lambda_0\geq\lambda_1\geq\ldots)$ (where
$\lambda_N=0$ for $N\gg0$) the corresponding fixed point is the
ideal $J_\lambda=\BC[z]\cdot (\BC y^0z^{\lambda_0}\oplus \BC
y^1z^{\lambda_1}\oplus\cdots)$. We will view $J_\lambda$ as an ideal
in $\CO_{\bC\times\bX}$ coinciding with $\CO_{\bC\times\bX}$ in a
neighborhood of infinity.

\emph{Notation:} We say $\lambda\supset\mu$ if $\lambda_i\geq\mu_i$
for any $i\geq0$. We say $\lambda\widetilde\supset\mu$ if
$\lambda_i\geq\mu_{i+1}$ for any $i\geq0$.

\noindent
 Consider a collection $\blambda=(\lambda^{kl})_{1\leq k,l\leq n}$ of
Young diagrams satisfying the following conditions:
\begin{equation}
\label{pool}
\lambda^{11}\supset\lambda^{21}\supset\cdots\supset\lambda^{n1}
\widetilde\supset\lambda^{11};\
\lambda^{22}\supset\lambda^{32}\supset\cdots\supset\lambda^{12}
\widetilde\supset\lambda^{22};\ \ldots;\
\lambda^{nn}\supset\lambda^{1n}\supset\cdots\supset\lambda^{n-1,n}
\widetilde\supset\lambda^{nn}\\
\end{equation}
We set $d_k(\blambda)=\sum_{l=1}^n|\lambda^{kl}|$, and
$\ul{d}(\blambda)=(d_0(\blambda):=d_n(\blambda),\ldots,d_{n-1}(\blambda))$.

Given such a collection $\blambda$ we define a parabolic sheaf
$\CF_\bullet=\CF_\bullet(\blambda)$, or just $\blambda$ by an abuse of
notation, as follows: for $1\leq k\leq n$ we set
\begin{equation}
\label{swim}
\CF_{k-n}=\bigoplus_{1\leq l\leq k} J_{\lambda^{kl}}w_l\oplus
\bigoplus_{k<l\leq n}J_{\lambda^{kl}}(-\bD_0)w_l
\end{equation}

 The following result is Lemma 3.3 of~\cite{ffnr}:
\begin{lem}
\label{evi}
The correspondence $\blambda\mapsto\CF_\bullet(\blambda)$ is a bijection
between the set of collections $\blambda$ satisfying~(\ref{pool})
such that $\ul{d}(\blambda)=\ul{d}$,
and the set of $\widetilde{T}\times\BC^*\times\BC^*$-fixed points in
$\CP_{\ul{d}}$.
\end{lem}

\subsection{Another realization of parabolic sheaves}
\label{realization} We will now introduce a different realization of
parabolic sheaves, and another parametrization of the fixed point
set which is very closely related to this new realization. This
construction originates from the work of Biswas~\cite{bi}. Let
$\sigma:\bC\times \bX\rightarrow \bC\times \bX$ denote the map
$\sigma(z,y)=(z,y^n)$, and let $G=\BZ/n\BZ$. Then $G$ acts on
$\bC\times \bX$ by multiplying the coordinate on $\bX$ with the
$n-$th roots of unity.

A parabolic sheaf $\CF_\bullet$ is completely determined by the flag of sheaves
$$
\CF_0(-\bD_0)\subset \CF_{-n+1}\subset\cdots\subset \CF_0,
$$
satisfying conditions ~\ref{PS}(a--e). To $\CF_\bullet$ we can
associate a single, $G$-invariant sheaf $\tilde \CF$ on $\bC\times
\bX$:
$$
\tilde
\CF=\sigma^*\CF_{-n+1}+\sigma^*\CF_{-n+2}(-\bD_0)+\cdots+\sigma^*\CF_0(-(n-1)\bD_0).
$$
This sheaf will have to satisfy certain numeric and framing
conditions that mimick conditions~\ref{PS}(b)--(e) (they are
explicitly written in~\cite{ffnr}). Conversely, any $G$-invariant
sheaf $\tilde \CF$ that satisfies those numeric and framing
conditions will
determine a unique parabolic sheaf. \\

If $\CF_\bullet$ is a $\widetilde{T}\times\BC^*\times\BC^*$ fixed
parabolic sheaf corresponding to a collection $\blambda$ as in the
previous section, then we have
\begin{equation} \label{explicit tilde fixed points}
\tilde \CF = \bigoplus_{l=1}^n J_{\lambda^l}(-(l-1)\bD_0)w_l,
\end{equation}
where $(\lambda^1,\ldots,\lambda^n)$ is a collection of partitions,
given by
\begin{equation} \label{relation lambdas}
\lambda^l_{ni-n\lfloor \frac {k-l}n \rfloor +k-l}=\lambda^{kl}_i.
\end{equation}
Here $\lfloor \frac {k-l}n \rfloor$ stands for the maximal integer smaller
than or equal to $\frac {k-l}n$.

For $j\in \BZ$, let $(j \textrm{ mod }n)$ denote an element of
$\{1,\ldots,n\}$ which is congruent to $j$ modulo $n$. For $i\geq
j\in \BZ$, we define
\begin{equation} \label{relation d lambda}
d_{ij}=\lambda^{j \textrm{ mod } n}_{i-j}
\end{equation}
This construction provides a collection
$(d_{ij})=\widetilde{\ul{d}}=\widetilde{\ul{d}}(\blambda)$ of
non-negative integers with the properties that
\begin{equation}
\label{fei}
d_{kj}\geq d_{ij}\ \forall i\geq k\geq j;\
d_{i+n,j+n}=d_{ij}\ \forall i\geq j;\
d_{ij}=0\ \operatorname{for}\ i-j\gg0.
\end{equation}
For $1\leq k\leq n$, we have
$$
d_k(\widetilde{\ul{d}})=\sum_{j\leq k}d_{kj}=\sum_{l=1}^n
\sum_{i\leq \lfloor \frac {k-l}n\rfloor } d_{k,l+ni}= \sum_{l=1}^n
\sum_{i\geq 0} \lambda_{ni-n\lfloor \frac {k-l}n\rfloor
+k-l}^l=\sum_{l=1}^{n}\sum_{i\geq 0} \lambda^{kl}_i=d_k(\blambda).
$$
Summarizing the above discussion, we have:
\begin{lem}
\label{dent}
The correspondence $\blambda\mapsto\widetilde{\ul{d}}(\blambda)$ is a
bijection between the set of collections $\blambda$ satisfying~(\ref{pool}),
and the set $D$ of collections $\widetilde{\ul{d}}$ satisfying~(\ref{fei}).
We have $\ul{d}(\blambda)=\ul{d}(\widetilde{\ul{d}}(\blambda))$.
\end{lem}

By virtue of Lemmas~\ref{evi} and~\ref{dent} we will parameterize
and sometimes denote the $\widetilde{T}\times\BC^*\times\BC^*$-fixed
points in $\CP_{\ul{d}}$ by collections $\widetilde{\ul{d}}$ such
that $\ul{d}=\ul{d}(\widetilde{\ul{d}})$.

\medskip
 \emph{Notation:} In what follows, given a collection
 $\widetilde{\ul{d}}$ as above we will denote by
 $\widetilde{\ul{d}}+\delta_{i,j}$ the collection
 $\widetilde{\ul{d}}{}'$, such that
 $\widetilde{\ul{d}}{}'_{i+ns,j+ns}=\widetilde{\ul{d}}_{i,j}+1\ (\forall s\in \BZ)$,
 while $\widetilde{\ul{d}}{}'_{p.q}=\widetilde{\ul{d}}_{p,q}$ for
 all other $(p,q)$.

\subsection{Correspondences}
If the collections $\ul{d}$ and $\ul{d}'$ differ at the only place
$i\in I:=\BZ/n\BZ$, and $d'_i=d_i+1$, then we consider the
correspondence
$\sE_{\ul{d},i}\subset\CP_{\ul{d}}\times\CP_{\ul{d}'}$ formed by the
pairs $(\CF_\bullet,\CF'_\bullet)$ such that for $j\not\equiv
i\pmod{n}$ we have $\CF_j=\CF'_j$, and for $j\equiv i\pmod{n}$ we
have $\CF'_j\subset\CF_j$. It is a smooth quasiprojective algebraic
variety of dimension $2\sum_{i\in I}d_i+1$.

We denote by $\bp$ (resp. $\bq$) the natural projection
$\sE_{\ul{d},i}\to\CP_{\ul{d}}$ (resp.
$\sE_{\ul{d},i}\to\CP_{\ul{d}'}$). For $j\equiv i\pmod{n}$ the
correspondence $\sE_{\ul{d},i}$ is equipped with a natural line
bundle $\sL_j$ whose fiber at $(\CF_\bullet,\CF'_\bullet)$ equals
$\Gamma(\bC,\CF_j/\CF'_j)$. Finally, we have a transposed
correspondence
$^\sT\sE_{\ul{d},i}\subset\CP_{\ul{d}'}\times\CP_{\ul{d}}$.

\subsection{Direct sum of equivariant $K$-groups}
We denote by ${}'V$ the direct sum of equivariant (complexified)
$K$-groups:
$${}'V=\oplus_{\ul{d}}K^{\widetilde{T}\times\BC^*\times\BC^*}
(\CP_{\ul{d}}).$$ It is a module over
$K^{\widetilde{T}\times\BC^*\times\BC^*}(pt)
=\BC[\widetilde{T}\times\BC^*\times\BC^*]=\BC[x_1,\ldots,x_n,v,u]$.
Here $u$ corresponds to a character $(x_1,\ldots,x_n, v,u)\mapsto
u$. We define $$V=\
{}'V\otimes_{K^{\widetilde{T}\times\BC^*\times\BC^*}(pt)}
\on{Frac}(K^{\widetilde{T}\times\BC^*\times\BC^*}(pt)).$$
 It is graded by $V=\oplus_{\ul{d}} V_{\ul{d}},\
V_{\ul{d}}=K^{\widetilde{T}\times\BC^*\times\BC^*}(\CP_{\ul{d}})
\otimes_{K^{\widetilde{T}\times\BC^*\times\BC^*}(pt)}
\on{Frac}(K^{\widetilde{T}\times\BC^*\times\BC^*}(pt))$.


\subsection{Action of a quantum affine group on $V$}
\label{aff op} The grading and the
correspondences $^\sT\sE_{\ul{d},i},\sE_{\ul{d},i}$ give rise to the
following operators on $V$
(note that though $\bp$ is not proper, $\bp_*$ is well defined on
the localized equivariant $K$-theory due to the finiteness of the
fixed point set of $\widetilde{T}\times\BC^*\times\BC^*$):

\begin{equation}
\label{3.1}
\fk_i=t_{i+1}^{-1}t_iu^{-\delta_{i,n}}v^{-2d_i+d_{i-1}+d_{i+1}-1}:\
V_{\ul{d}}\to V_{\ul{d}}
\end{equation}
\begin{equation}
\label{3.2} \fe_i=t_{i+1}^{-1}v^{d_{i+1}-d_i-i+1}\bp_*\bq^*:\
V_{\ul{d}}\to V_{\ul{d}-i}
\end{equation}
\begin{equation}
\label{3.3}
\ff_i=-t_i^{-1}u^{-\delta_{i,n}}v^{d_i-d_{i-1}+i}\bq_*(L_{i-n}\otimes\bp^*):\
V_{\ul{d}}\to V_{\ul{d}+i}
\end{equation}

\medskip
According to the Conjecture 3.7 of~\cite{bf} the following theorem
holds:\footnote{\ Actually, (\ref{3.1}--\ref{3.3}) differ from
formulas in~\cite{bf} by a slight rescaling. We prefer those, since
they are simpler.}
\begin{thm}
\label{braval} For $n>2$, these operators $\fk_i, \fe_i, \ff_i \
(1\leq i\leq n)$ satisfy the relations in $U_v(\widehat{\fsl_n})$,
i.e. they give rise to an action of a quantum affine group
$U_v(\widehat{\fsl_n})$ on $V$.
\end{thm}

\textsl{Since the fixed point basis of $M$ corresponds to the
Gelfand-Tsetlin basis of the universal Verma module over
$U_v(\fgl_n)$, we propose to call the fixed point basis of $V$ the}
{\em affine Gelfand-Tsetlin basis}.

\subsection{Quantum toroidal algebra}
\label{guay} Let $(a_{kl})_{1\leq k,l\leq n}=\widehat{A}_{n-1}$
stand for the Cartan matrix of $\widehat{\fsl}_n$. The double affine
loop algebra $U'_v(\widehat{\fsl_n})$ is an associative algebra over
$\mathbb Q(v)$ generated by $e_{k,r}$, $f_{k,r}$, $v^{\pm h_k}$,
$h_{k,m}$ $(1\leq k\leq n, r \in \BZ,  m\in \BZ\setminus \{0\})$
with the relations~(\ref{1}--\ref{6}), where $k,l$ are understood as
residues modulo $n$, so that for instance if $k=n$ then $k+1=1$.

The quantum toroidal algebra \"{U}$_v(\widehat{\fsl}_n)$ is an
associative algebra over $\mathbb C(u,v)$ generated by $e_{k,r}$,
$f_{k,r}$, $v^{\pm h_k}$, $h_{k,m}$ $(1\leq k\leq n, r \in \BZ, m\in
\BZ\setminus \{0\})$ with the same relations as in
$U'_v(\widehat{\fsl_n})$ except for relations~(\ref{2},~\ref{5}) for
the pairs $(k,l)=(1,n),(n,1)$. These relations are modified as
follows. We introduce the shifted generating series
$\hat{x}_n^{\pm}(z):=x_n^{\pm}(zv^nu^2),\
\hat{\psi}_n^{\pm}(z)=\psi_n^{\pm}(zv^nu^2)$.

 Now the new relations read
\begin{equation}
\label{tor4} \hat{x}_n^{\pm}(z)x_1^{\pm}(w)(z-v^{\mp 1}w)=(v^{\mp
1}z-w)x_1^{\pm}(w)\hat{x}_n^{\pm}(z),
\end{equation}

\begin{equation}
\label{tor2.1} \hat{\psi}_n^s(z)x_1^{\pm}(w)(z-v^{\mp
1}w)=x_1^{\pm}(w) \hat{\psi}_n^s(z)(v^{\mp 1}z-w),
\end{equation}

\begin{equation}
\label{tor2.2} \psi_1^s(z)\hat{x}_n^{\pm}(w)(z-v^{\mp
1}w)=\hat{x}_n^{\pm}(w) \psi_1^s(z)(v^{\mp 1}z-w).
\end{equation}

Thus we have
$U'_v(\widehat{\fsl_n})=$\"{U}$_v(\widehat{\fsl}_n)$$/(v^nu^2=1)$.

Note that \"{U}$_v(\widehat{\fsl}_n)$ coincides with
\"{U}$'$--modification of \"{U} introduced in~\cite{vv2}, with $d$
not specialized to a complex number and with the central element
$c=1$, via the isomorphism \"{U}$_v(\widehat{\fsl}_n)$
$\overset{\Phi}\to$~\"{U}$'$, such that $\Phi(v)=v$ and
$\Phi(u)=d^{\frac{n}{2}}v^{-\frac{n}{2}}$. It is defined on the
generating series as
$$\Phi(x_i^{+}(z))=\be_{i-1}^{\pm}(d^{-i}z),\ \Phi(x_i^{-}(z))=\bff_{i-1}^{\pm}(d^{-i}z),\  \Phi(\psi_i^{\pm}(z))=\bk_{i-1}^{\pm}(d^{-i}z).$$

\subsection{Main theorem}
\label{quotients} For any $m<i\in\BZ$ we will denote by
$\ul{\CW}{}_{mi}$ the quotient $\ul{\CF}{}_i/\ul{\CF}{}_m$ of the
tautological vector bundles, living on
$\CP_{\ul{d}}\times\bC\subset\CP_{\ul{d}}\times\bS$. Once again,
$\pi:\CP_{\ul{d}}\times(\bC\backslash\{\infty\})\to\CP_{\ul{d}}$
denotes the standard projection. Let us consider the generating
series:
$$\bb_{mi}(z):=\Lambda^{\bullet}_{-1/z}(\pi_*(\ul{\CW}{}_{mi}\mid_{\bC\backslash \{\infty\}})):\ V_{\ul{d}}\to V_{\ul{d}}[[z^{-1}]]$$

\begin{cor}
\label{faktory} The expression
$\bb_{mi}(zv^{-i-2})^{-1}\bb_{mi}(zv^{-i})^{-1}\bb_{m,i-1}(zv^{-i})\bb_{m,i+1}(zv^{-i-2})$
is independent of the choice of $m$.
\end{cor}

 Proof is analogous to the proof of Corollary~\ref{faktor}.

\bigskip

We will denote by $\psi_i^{\pm}(z)$ the common value of the
expressions
\begin{equation}
\label{psishki}
t_{i+1}^{-1}t_iu^{-\delta_{i,n}}v^{d_{i+1}-2d_i+d_{i-1}-1}
\left(\bb_{m,i-n}(zv^{-i-2})^{-1}\bb_{m,i-n}(zv^{-i})^{-1}\bb_{m,i-1-n}(zv^{-i})\bb_{m,i+1-n}(zv^{-i-2})\right)^{\pm}.
\end{equation}

Recall that $v$ stands for the character of
$\widetilde{T}\times\BC^*\times\BC^*:\ (\ul{t},v,u)\mapsto v$. We
define the line bundle $\sL'_k:=v^k\sL_k$ on the correspondence
$\sE_{\ul{d},k}$, that is $\sL'_k$ and $\sL_k$ are isomorphic as
line bundles but the equivariant structure of $\sL'_k$ is obtained
from the equivariant structure of $\sL_k$ by the twist by the
character $v^k$.

For $1\leq k\leq n$ we consider the following generating series of
operators on $V$:

\begin{equation}
\label{raz1.1} \psi_k^{\pm}(z)=:\sum_{r=0}^{\pm \infty}
\psi^{\pm}_{k,r}z^{\mp r}:\ V_{\ul{d}} \to V_{\ul{d}}[[z^{\mp 1}]]
\end{equation}

\begin{equation}
\label{dvas1} x_{k}^{+}(z)=\sum_{r=-\infty}^\infty e_{k,r}z^{-r}:\
V_{\ul{d}}\to V_{\ul{d}-k}[[z,z^{-1}]]
\end{equation}

\begin{equation}
\label{tris1} x_{k}^{-}(z)=\sum_{r=-\infty}^\infty f_{k,r}z^{-r}:\
V_{\ul{d}}\to V_{\ul{d}+k}[[z,z^{-1}]]
\end{equation}

\begin{equation}
\label{ddvas1}
e_{k,r}:=t_{k+1}^{-1}v^{d_{k+1}-d_k+1-k}\bp_*((L'_{k-n})^{\otimes
r}\otimes \bq^*):\ V_{\ul{d}}\to V_{\ul{d}-k}
\end{equation}

\begin{equation}
\label{ttris1}
f_{k,r}:=-t_k^{-1}u^{-\delta_{k,n}}v^{d_k-d_{k-1}+k}\bq_*(L_{k-n}\otimes
(L'_{k-n})^{\otimes r}\otimes \bp^*):\ V_{\ul{d}}\to V_{\ul{d}+k}
\end{equation}

\begin{thm}
\label{varaf} These generating series of operators $\psi_{k}^\pm(z),
x_{k}^\pm(z)$ on $V$ defined in~(\ref{psishki}--\ref{ttris1})
satisfy the relations in \"{U}$_v(\widehat{\fsl}_n)$, i.e. they give
rise to an action of \"{U}$_v(\widehat{\fsl}_n)$ on $V$.
\end{thm}


First, we compute the matrix coefficients of operators $e_{i,r},
f_{i,r}$ and the eigenvalues of $\psi_i^{\pm}(z)$. For accomplishing
this goal we need to know the torus character in the tangent space
to $\sE_{\ul{d},i}$ (and $\CP_{\ul{d}}$) at the torus fixed point
given by indices $\widetilde{\ul{d}},\widetilde{\ul{d'}}$ (and
$\widetilde{\ul{d}}$ correspondingly). These characters are computed
in~\cite{ffnr} (see Propositions 4.15, 4.21 and Remark 4.17 of {\em
loc. cit.}):

\begin{prop} \label{character to correspondence}
a) The torus character in the tangent space to $\sE_{\ul{d},i}$ at
the torus fixed point given by indices
$\widetilde{\ul{d}},\widetilde{\ul{d'}}$ equals
$$
\sum_{k=1}^n \sum_{l\leq k}^{l'\leq k-1} \frac{t^2_l}{t^2_{l'}}\cdot
v^2\frac {({v}^{2d_{(k-1)l'}}-1)({v}^{-2d_{kl}}-1)}{v^2-1}\cdot
{u}^{2\lfloor \frac {-l'}n\rfloor-2\lfloor \frac
{-l}n\rfloor}+\sum_{k=1}^n \sum_{l'\leq k-1}
\frac{t^2_k}{t^2_{l'}}\cdot v^2\frac
{{v}^{2d_{(k-1)l'}}-1}{v^2-1}\cdot {u}^{2\lfloor \frac
{-l'}n\rfloor-2\lfloor \frac {-k}n\rfloor}-
$$
$$
-\sum_{k=1}^n \sum_{l\leq k}^{l'\leq k}\frac{t^2_l}{t^2_{l'}}\cdot
v^2\frac {({v}^{2d_{kl'}}-1)({2}^{-2d_{kl}}-1)}{v^2-1}\cdot
{u}^{2\lfloor \frac {-l'}n\rfloor-2\lfloor \frac
{-l}n\rfloor}-\sum_{k=1}^n \sum_{l\leq k} \frac{t^2_l}{t^2_k}\cdot
v^2\frac {{v}^{-2d_{kl}}-1}{v^2-1}\cdot {u}^{2\lfloor \frac
{-k}n\rfloor-2\lfloor \frac {-l}n\rfloor}+
$$
$$
+v^2-{v}^{-2d_{ij}+2d_{(i-1)j}}+\frac{t^2_j}{t^2_i}\cdot
{v}^{-2d_{ij}+2d_{ii}}\cdot {u}^{2\lfloor \frac
{-i}n\rfloor-2\lfloor \frac {-j}n\rfloor}+\sum_{j\neq k\leq
i-1}\frac{t^2_j}{t^2_k}\cdot {v}^{-2d_{ij}}\cdot
({v}^{2d_{ik}}-{v}^{2d_{(i-1)k}})\cdot {u}^{2\lfloor \frac
{-k}n\rfloor-2\lfloor \frac {-j}n\rfloor}
$$
if $\widetilde{\ul{d}}{}'=\widetilde{\ul{d}}+\delta_{i,j}$ for
certain $j\leq i$.


b) The torus character in the tangent space to $\CP_{\ul{d}}$ at the
torus fixed point $\widetilde{\ul{d}}$ equals
$$ \sum_{k=1}^n \sum_{l\leq k}^{l'\leq k-1}
\frac{t^2_l}{t^2_{l'}}\cdot v^2\frac
{({v}^{2d_{(k-1)l'}}-1)({v}^{-2d_{kl}}-1)}{v^2-1}\cdot {u}^{2\lfloor
\frac {-l'}n\rfloor-2\lfloor \frac {-l}n\rfloor}+\sum_{k=1}^n
\sum_{l'\leq k-1} \frac{t^2_k}{t^2_{l'}}\cdot v^2\frac
{{v}^{2d_{(k-1)l'}}-1}{v^2-1}\cdot {u}^{2\lfloor \frac
{-l'}n\rfloor-2\lfloor \frac {-k}n\rfloor}-
$$$$
-\sum_{k=1}^n \sum_{l\leq k}^{l'\leq k}\frac{t^2_l}{t^2_{l'}}\cdot
v^2\frac {({v}^{2d_{kl'}}-1)({v}^{-2d_{kl}}-1)}{v^2-1}\cdot
{u}^{2\lfloor \frac {-l'}n\rfloor-2\lfloor \frac
{-l}n\rfloor}-\sum_{k=1}^n \sum_{l\leq k} \frac{t^2_l}{t^2_k}\cdot
v^2\frac {{v}^{-2d_{kl}}-1}{v^2-1}\cdot {u}^{2\lfloor \frac
{-k}n\rfloor-2\lfloor \frac {-l}n\rfloor}$$

\end{prop}

So analogously to Theorem 3.17~(\cite{ffnr}) we get the following
proposition

\begin{prop}
\label{analog}
 \textit{Define
$p_{i,j}:=t_{j\pmod{n}}^2v^{-2d_{ij}}u^{-2\lfloor\frac{-j+n}{n}\rfloor}=t_{j\pmod{n}}^2v^{-2d_{ij}}u^{2\lceil\frac{j-n}{n}\rceil}.$}

a) The matrix coefficients of the operators $f_{i,r}, e_{i,r}$ in
the fixed point basis $\{[\widetilde{\ul{d}}]\}$ of $V$ are as
follows:
$$f_{{i,r}[\widetilde{\ul{d}},\widetilde{\ul{d}}{}']}=
-t_i^{-1}u^{-\delta_{i,n}}v^{d_i-d_{i-1}+i} p_{i,j}(p_{i,j}v^i)^r
(1-v^2)^{-1}\prod_{j\ne k\leq i}(1-p_{i,j}p_{i,k}^{-1})^{-1}
\prod_{k\leq i-1}(1-p_{i,j}p_{i-1,k}^{-1})$$ if
$\widetilde{\ul{d}}{}'=\widetilde{\ul{d}}+\delta_{i,j}$ for certain
$j\leq i$;
$$e_{{i,r}[\widetilde{\ul{d}},\widetilde{\ul{d}}{}']}=
t_{i+1}^{-1}v^{d_{i+1}-d_i+1-i}(p_{i,j}v^{i+2})^r(1-v^2)^{-1}\prod_{j\ne
k\leq i}(1-p_{i,k}p_{i,j}^{-1})^{-1} \prod_{k\leq
i+1}(1-p_{i+1,k}p_{i,j}^{-1})$$ if
$\widetilde{\ul{d}}{}'=\widetilde{\ul{d}}-\delta_{i,j}$ for certain
$j\leq i$.

All the other matrix coefficients of $e_{i,r},f_{i,r}$ vanish.

\noindent b) The eigenvalue of $\psi_i^{\pm}(z)$ on
$\{[\widetilde{\ul{d}}]\}$ equals
$$ \frac{t_iv^{d_{i+1}-2d_{i}+d_{i-1}-1}}{t_{i+1}u^{\delta_{i,n}}}\prod_{j\le
i}(1-z^{-1}v^{i+2}p_{i,j})^{-1}(1-z^{-1}v^ip_{i,j})^{-1}\prod_{j\le
i+1}(1-z^{-1}v^{i+2}p_{i+1,j})\prod_{j\le
i-1}(1-z^{-1}v^ip_{i-1,j}),
$$
where it is expanded in $z^{\mp1}$ depending on the sign $\pm$.

\end{prop}

\begin{rem}{\em These formulas are the same as in
Proposition~\ref{matrix_elements} with the change
$s_{i,j}\rightsquigarrow p_{i,j}$ and the factor $u^{-\delta_{i,n}}$
appearing in $\psi_i^\pm(z), f_i(z)$.}
\end{rem}


\begin{proof} [Proof of Theorem~\ref{varaf}.]
 For any $k\in\BZ$ we define $x_k^\pm(z),\psi_k^{\pm}(z)$ by the same
formulas~(\ref{psishki}--\ref{ttris1}) with $\delta_{k,n}$ being
changed to $\delta_{k\pmod{n},0}$.

  First, because of the above remark and our computational proof of Theorem~\ref{var}, relations (\ref{1}--\ref{6}) still hold.
Indeed, relations (\ref{4}--\ref{6}) are verified along the same
lines with just $p_{i,j}$ instead of $s_{i,j}$. Similarly with
(\ref{1}--\ref{2}). The only nontrivial equality is $\psi_{i,0}^+ -
\psi_{i,0}^-=\chi_{i,0}$, where $\chi_{i,0}$ is defined in the same
way with $p_{ij}$'s instead of $s_{ij}$'s. However, it is a
statement of Theorem~\ref{braval}.\footnote{\ Actually, it reduces
to the equality from the proof of Proposition 2.21,~\cite{bf}. The
point why $u^{-\delta_{i,n}}$ appears now is that $\prod_{j\leq
i+1}p_{i+1,j}\prod_{j\leq
i}p_{i,j}^{-1}=t_{i+1}^2u^{2\lceil\frac{i+1-n}{n}\rceil}v^{2d_{i}-2d_{i+1}}$,
while for $s_{i,j}$ we had the same equality without
$u^{2\lceil\frac{i+1-n}{n}\rceil}$.} The relation (\ref{3}) follows.

 So the only thing left is to verify relations (\ref{tor4}--\ref{tor2.2}).
 Let us point out that $p_{i+n,j+n}=u^2p_{i,j}$ for all $i, j$.
 Hence formulas of Proposition~\ref{analog} imply that for any $k\in \ZZ$:
$$\psi_{k}^\pm(z)=\psi_{k+n}^\pm(v^nu^2z),\ x_k^+(z)=v^n\cdot x_{k+n}^+(v^nu^2z),\ x_k^-(z)=v^{-n}u^{-2}\cdot x_{k+n}^-(v^nu^2z).$$
 In particular, we get
$$\hat{\psi}_n^{\pm}(z)=\psi_0^{\pm}(z),\  \hat{x}_n^+(z)=v^{-n}x_0^+(z),\ \hat{x}_n^-(z)=v^nu^2x_0^-(z).$$
 Now relations~(\ref{tor4}--\ref{tor2.2}) follow again from
Theorem~\ref{var} and the above remark.
\end{proof}

\subsection{Specialization of Gelfand-Tsetlin basis}
We fix a positive integer $K$ (a level). We consider an $n$-tuple
$\mu=(\mu_{1-n},\ldots,\mu_0)\in\BZ^n$ such that
$\mu_0+K\geq\mu_{1-n}\geq\mu_{2-n}\geq\ldots\geq\mu_{-1}\geq\mu_0$.
We view $\mu$ as a dominant (integrable) weight of
$\widehat{\fgl}_n$ of level $K$. We extend $\mu$ to a nonincreasing
sequence $\widetilde{\mu}=(\widetilde{\mu}_i)_{i\in\BZ}$ setting
$\widetilde{\mu}_i:=\mu_{i\pmod{n}}+\lfloor\frac{-i}{n}\rfloor K$.

We define a subset $D(\mu)$ ({\em affine Gelfand-Tsetlin patterns})
of the set $D$ of all collections $\widetilde{\ul{d}}$ satisfying
the conditions~(\ref{fei}) as follows:
\begin{equation}
\label{ryb} \widetilde{\ul{d}}\in D(\mu)\ \on{iff}\
d_{ij}-\widetilde{\mu}_j\leq d_{i+l,j+l}-\widetilde{\mu}_{j+l}\
\forall\ j\leq i,\ l\geq0.
\end{equation}

We specialize the values of $t_1,\ldots,t_n,v,u$ so that
\begin{equation}
\label{rybn} u=v^{-K-n},\ t_j=v^{\widetilde{\mu}_j-j+1}.
\end{equation}

We define the renormalized vectors
\begin{equation}
\label{renorm} \langle\widetilde{\ul{d}}\rangle:=
C^{-1}_{\widetilde{\ul{d}}}[\widetilde{\ul{d}}]
\end{equation}

where $C_{\widetilde{\ul{d}}}$ is the product $\prod_{w\in
T_{\widetilde{\ul{d}}}\CP_{\ul{d}}}(1-w)$ and $w$ runs over all
$\widetilde{T}\times\BC^*\times\BC^*$-weights in the tangent space
to $\CP_{\ul{d}}$ at the point $\widetilde{\ul{d}}$. The explicit
formula for the multiset $\{w\}$ is provided by
Proposition~\ref{character to correspondence}b).


\begin{prop}
\label{GTaff}  The only nonzero matrix coefficients of the operators
$f_{i,r}, e_{i,r}$ in the renormalized fixed point basis
$\{\langle\widetilde{\ul{d}}\rangle\}$ of $V$ are as follows:
$$e_{{i,r}\langle\widetilde{\ul{d}}+\delta_{i,j},\widetilde{\ul{d}}\rangle}=
t_{i+1}^{-1}v^{d_{i+1}-d_i-i}(p_{i,j}v^{i})^r
(1-v^2)^{-1}\prod_{j\ne k\leq i}(1-p_{i,j}p_{i,k}^{-1})^{-1}
\prod_{k\leq i-1}(1-p_{i,j}p_{i-1,k}^{-1}),$$
$$f_{{i,r}\langle\widetilde{\ul{d}}-\delta_{i,j},\widetilde{\ul{d}}\rangle}=-t_i^{-1}u^{-\delta_{i,n}}v^{d_i-d_{i-1}-1+i}
p_{i,j}v^2(p_{i,j}v^{i+2})^r(1-v^2)^{-1}\prod_{j\ne k\leq
i}(1-p_{i,k}p_{i,j}^{-1})^{-1} \prod_{k\leq
i+1}(1-p_{i+1,k}p_{i,j}^{-1}).$$


\end{prop}

\begin{proof}

According to Proposition~\ref{analog}, matrix coefficients
${e_{i,r}}_{[\widetilde{\ul{d}}',\widetilde{\ul{d}}]}\ (
{f_{i,r}}_{[\widetilde{\ul{d}}',\widetilde{\ul{d}}]})$ are nonzero
only if $\widetilde{\ul{d}}'=\widetilde{\ul{d}}+\delta_{i,j}\
(\widetilde{\ul{d}}'=\widetilde{\ul{d}}-\delta_{i,j})$ for some
$j\leq i$. In those cases they are given by the Bott-Lefschetz fixed
point formula:
$$ e_{i,r
[\widetilde{\ul{d}}',\widetilde{\ul{d}}]}=
t_{i+1}^{-1}v^{d_{i+1}-d_i-i}(t_j^2v^{-2d_{ij}}u^{2\lceil\frac{j-n}{n}\rceil}v^i)^r
\frac {\dsp \prod_{w\in T_{\widetilde{\ul{d}}'}\CP_{\ul{d}'}}
(1-w)}{\dsp \prod_{w\in T_{(\widetilde{\ul{d}},\widetilde{\ul{d}'}
)}\sE_{\ul{d},i}} (1-w)};
$$
$$ f_{i,r
[\widetilde{\ul{d}}',\widetilde{\ul{d}}]}=
-t_{i}^{-1}u^{-\delta_{i,n}}v^{d_{i}-d_{i-1}-1+i}(t_j^2v^{-2d_{ij}+2}u^{2\lceil\frac{j-n}{n}\rceil})(t_j^2v^{-2d_{ij}+2}u^{2\lceil\frac{j-n}{n}\rceil}v^i)^r
\frac {\dsp \prod_{w\in T_{\widetilde{\ul{d}}'}\CP_{\ul{d}'}}
(1-w)}{\dsp \prod_{w\in T_{(\widetilde{\ul{d}},\widetilde{\ul{d}'}
)}\sE_{\ul{d},i}} (1-w)}.
$$

 So after renormalizing vectors according to (\ref{renorm}) we have:
$$e_{{i,r}\langle\widetilde{\ul{d}'}{},\widetilde{\ul{d}}\rangle}=
  -f_{{i,r}[\widetilde{\ul{d}}{},\widetilde{\ul{d}'}]}t_it_{i+1}^{-1}u^{\delta_{i,n}}v^{d_{i+1}-2d_i+d_{i-1}-2i}(t_j^2v^{-2d_{ij}}u^{2\lceil\frac{j-n}{n}\rceil})^{-1},$$
$$f_{{i,r}\langle\widetilde{\ul{d}}{},\widetilde{\ul{d}'}\rangle}=
  -e_{{i,r}[\widetilde{\ul{d}'}{},\widetilde{\ul{d}}]}t_i^{-1}t_{i+1}u^{-\delta_{i,n}}v^{-d_{i+1}+2d_i-d_{i-1}+2i}(t_j^2v^{-2d_{ij}}u^{2\lceil\frac{j-n}{n}\rceil}).$$
 Now, the statement follows from Proposition~\ref{analog}.
\end{proof}

We define $V(\mu)$ as the $\BC(v)$--linear span of the vectors
$\langle\widetilde{\ul{d}}\rangle$ for $\widetilde{\ul{d}}\in
D(\mu)$.

\begin{thm}
\label{rybni} Formulas of Theorem~\ref{varaf} give rise to the
action of \"{U}$_v(\widehat{\fsl}_n)/(u-v^{-K-n})$ in $V(\mu)$.
\end{thm}

\begin{proof}

 Analogously to Theorem 3.23,~\cite{ffnr}, we have to check two things:

\noindent
(i) for $\widetilde{\ul{d}}\in D(\mu)$ the denominators of
the matrix coefficients
$e_{i,r\langle\widetilde{\ul{d}},\widetilde{\ul{d}}{}'\rangle},\
f_{i,r\langle\widetilde{\ul{d}},\widetilde{\ul{d}}{}'\rangle}$ do
not vanish;

\noindent
(ii) for $\widetilde{\ul{d}}\in D(\mu),\
\widetilde{\ul{d}}{}'\not\in D(\mu)$ the numerators of the matrix
coefficients
$e_{i,r\langle\widetilde{\ul{d}},\widetilde{\ul{d}}{}'\rangle},\
f_{i,r\langle\widetilde{\ul{d}},\widetilde{\ul{d}}{}'\rangle}$ do
vanish.

  Both verifications are straightforward and we will sketch only those for $e_{i,r}$
  operators.\footnote{\ We choose to provide some details of the verification, since they were missing in~\cite{ffnr}.}

\noindent
 Under the above specialization, for $j=nj_0+j_1\ (j_0\in \ZZ, 1\leq j_1\leq n)$, we get
  $$p_{i,j}=v^{2\wt{\mu}_{j_1}-2j_1+2-2d_{i,j}-2j_0(K+n)}=v^{2(\wt{\mu}_j-j-d_{i,j}+1)}.$$

\medskip
\noindent
 (i) We need to show $\wt{\mu}_j-j-d_{i,j}\ne \wt{\mu}_k-k-d_{i,k}-1,\ \forall k\leq i,$ for
 $\underline{\wt{d}}\in D(\mu)$, such that $\underline{\wt{d}}-\delta_i^j\in D$.

\noindent
 $\circ$
  If $j\leq k\leq i$, then $d_{i,j}-\wt{\mu}_j\leq d_{i+k-j,k}-\wt{\mu}_k\leq d_{i,k}-\wt{\mu}_k$ and $j<k+1$, implying the result.

\noindent
 $\circ$
  If $k<j\leq i$, then $d_{i,k}-\wt{\mu}_k\leq d_{i+j-k,j}-\wt{\mu}_j\leq d_{i,j}-\wt{\mu}_j$ and $k+1\leq j$.
  This implies $d_{i,k}-\wt{\mu}_k+k+1\leq  d_{i,j}-\wt{\mu}_j+j$. However, if the equality happens above,
  then we have $j=k+1$ and $d_{i+j-k,j}=d_{i,j}$, that is $d_{i+1,j}=d_{i,j}$.
  But this contradicts our assumption $\underline{\wt{d}}-\delta_i^j\in  D$.

\medskip
\noindent
   (ii) We need to prove an existence of $k\leq i-1$ satisfying $\wt{\mu}_j-j-d_{i,j}=\wt{\mu}_k-k-d_{i-1,k}-1$
  for $\underline{\wt{d}}\in D(\mu)$, such that $\underline{\wt{d}}-\delta_i^j\in D\backslash D(\mu)$.

   Recalling the definition of $D(\mu)$, the latter condition on $\underline{\wt{d}}$ guarantees
  $d_{i-l,j-l}-\wt{\mu}_{j-l}=d_{i,j}-\wt{\mu}_j$ for some $l\geq  1$ and so $d_{i-1,j-1}-\wt{\mu}_{j-1}=d_{i,j}-\wt{\mu}_j$.
  Thus, picking  $k:=j-1$ works.
\end{proof}

Restricting $V(\mu)$ to the subalgebra of
\"{U}$_v(\widehat{\fsl}_n)$, generated by $\{e_{i,0}, f_{i,0},
v^{\pm h_i}\}_{1\leq i\leq n}$ which is isomorphic to
$U_v(\widehat{\fsl}_n)$ (called \textsl{horizontal} in~\cite{vv2})
we obtain the same named $U_v(\widehat{\fsl}_n)$-module with the
Gelfand-Tsetlin basis parameterized by $D(\mu)$. Recall that in the
proof of Theorem 3.22,~\cite{ffnr}, there was constructed a
bijection between $D(\mu)$ and Tingley's crystal $\fB_\mu$ of
cylindric plane partitions model of section 4~\cite{t}. This answers
Tingley's \textsl{Question 1} (\cite{t}, p.38).

Finally we formulate a conjecture:

\begin{conj}
  \"{U}$_v(\widehat{\fsl}_n)/(u-v^{-K-n})$--module $V(\mu)$ is
isomorphic to Uglov-Takemura module, constructed in~\cite{u1}.
\end{conj}

It seems likely that these \"{U}$_v(\widehat{\fsl}_n)$--modules are
obtained by the application of the {\em Schur} functor (\cite{gkv})
to the irreducible $\mathfrak X$-semisimple modules over the double
affine Cherednik algebra \"{H}$_n(v)$ of type $A_{n-1}$,
see~\cite{sv}.

\medskip

\end{document}